\documentclass{amsart}
\usepackage{style}

\begin{document}

\title[Alg. structures in comodule categories over WBAs]
{Algebraic structures in comodule categories over weak bialgebras}

\author{Chelsea Walton, Elizabeth Wicks, and Robert Won}

\address{Walton: Department of Mathematics, Rice University, Houston, TX 77005, USA}

\email{notlaw@rice.edu}

\address{Wicks: Microsoft Corporation, Redmond, WA 98052, USA} 
\email{elizabeth.wicks@microsoft.com}

 \address{Won: Department of Mathematics, The George Washington University, Washington, Dc 20009,~USA} 
\email{robertwon@gwu.edu}

\begin{abstract} 
For a bialgebra $L$ coacting on a $\kk$-algebra $A$, a classical result states $A$ is a right $L$-comodule algebra if and only if $A$ is an algebra in the category $\mathcal{M}^{L}$ of right $L$-comodules. We generalize this to the setting of weak bialgebras $H$. We prove that there is an isomorphism between the categories of right $H$-comodule algebras and the category of algebras in $\mathcal{M}^{H}$. We also recall and introduce the formulaic notion of $H$ coacting on a $\kk$-coalgebra and on a Frobenius $\kk$-algebra, respectively, and prove analogous category isomorphisms. Many examples are provided throughout.
\end{abstract}


\maketitle



\section{Introduction} \label{sec:intro}

This work is a study of quantum symmetry in the context of weak bialgebra and weak Hopf algebra (co)actions on algebraic structures, particularly with an aim of attracting those new to the subject.

\smallskip

Weak Hopf algebras and weak bialgebras  were introduced by B{\"o}hm, Nill, and Szlach{\'a}nyi  as generalizations of Hopf algebras and bialgebras over a field $\kk$ \cite{BNS, BoSz, Nill, Sz3}; these are sometimes referred to as quantum (semi)groupoids. A weak bialgebra is a $\kk$-vector space with algebra and coalgebra structures satisfying weaker compatibility conditions than those of a bialgebra [\cref{def:wba}], and a weak Hopf algebra is a weak bialgebra which admits an antipode [\cref{def:wha}]. The definitions of these structures were inspired by symmetry problems in quantum field theory and operator algebras; see, \cite{NSW, NV2000, SFR}. 
As noted in \cite{BNS} and \cite{NV}, key constructions of weak bialgebras include  groupoid algebras and their duals, dynamical twists of quantum groups \cite{EN},
Yamanouchi's generalized Kac algebras \cite{Yam}, and Hayashi's face algebras \cite{H93, Hayashi98}.
Weak Hopf algebras also arise naturally in the setting of tensor categories (where by a tensor category we mean an essentially small, $\kk$-linear, additive, abelian, finite, monoidal category with bi-additive tensor product). Indeed their categories of (co)modules have monoidal structures, 
and a remarkable result of Hayashi is that every fusion category is tensor equivalent to the category of finite-dimensional comodules over a face algebra (i.e., over a certain cosemisimple weak bialgebra) \cite[Theorem~4.1]{Hayashi99}, see also \cite[Theorem~1.4]{Sz2} \cite[Corollary~2.22]{ENO}.

\smallskip

Classically, symmetries of algebras are captured both by actions of groups by automorphisms and by actions of Lie algebras by derivations;
this framework works well when the algebras under investigation are commutative.
However, for noncommutative algebras, a broader notion of symmetry is needed, and actions and coactions of bialgebras and of Hopf algebras provide a suitable framework \cite{Drinfeld} \cite[Part~I]{Kassel}; see also \cite[Section~4]{Walton}. 
Typically the bialgebra and the algebra are connected graded or, more generally, augmented over $\kk$. We aim to establish a framework for quantum symmetry for algebras with a base algebra which is larger than the base field $\kk$. 
Such algebras include path algebras of quivers \cite[Chapter~II]{ASS2006}, 
smash product algebras \cite[Chapter 4]{Mo}, and skew Calabi-Yau algebras \cite{RRZ}. 
As described below, path algebras of quivers admit natural symmetries via coactions of Hayashi's face algebras, and so we expect that (co)actions of weak bialgebras and weak Hopf algebras will serve as a means to achieve this goal; see also \cite{HWWW}.

\smallskip

To begin, we need a {\it formulaic} definition of a weak bialgebra $H$ coacting on an algebra.

\medskip

\noindent {\bf Convention.} We employ Sweedler notation freely here and throughout (see Section~\ref{sec:modules}). 

\begin{definition}[{\cite[Definition 2.1]{Bohm}, see also \cite[Proposition 4.10]{C-DG}}]
Let $H$ be a weak bialgebra and $H_t$ be its target counital subalgebra [Definition~\ref{def:eps}].
We say that a $\kk$-algebra $(A, \; m:A \otimes A \to A, \; u:\kk \to A)$ is a {\it right $H$-comodule algebra} if $A$ is a right $H$-comodule via $\rho: A \to A \otimes H$, $a \mapsto a_{[0]} \otimes a_{[1]}$ so that $(ab)_{[0]} \otimes (ab)_{[1]} = a_{[0]}b_{[0]} \otimes a_{[1]}b_{[1]}$ and $\rho(1_A) \in A \otimes H_t$. Here, $m(a \otimes b)=ab$ and $1_A = u(1_\kk)$.
\end{definition}

As for bialgebras and Hopf algebras, the formulaic definition above should coincide with the notion of an algebra in the monoidal category $\mathcal{M}^H$ of comodules over $H$. The monoidal structure of $\mathcal{M}^H$ is understood by work of Nill and B\"ohm--Caenepeel--Janssen \cite{Nill, BCJ} (see Section~\ref{sec:corep}).
Notably, when $H$ is a weak bialgebra (that is not a bialgebra),
the monoidal product in $\mathcal{M}^H$ is not $\otimes_{\kk}$, nor is its unit object equal to $\kk$. So, algebras in $\mathcal{M}^H$ are not canonically $\kk$-algebras in this case. 
Nevertheless, there is a correspondence between 
right $H$-comodule algebras and algebras in $\mathcal{M}^H$;
this comprises our first  result. 


\begin{theorem}[Theorem~\ref{thm:alg}] \label{thm:alg-intro}
Let $H$ be a weak bialgebra. Consider the category $\mathcal{A}^H$ of right $H$-comodule algebras
(see Definition~\ref{def:A}) and the category ${\sf Alg}(\mathcal{M}^H)$ of algebras in the monoidal category $\mathcal{M}^H$ of right $H$-comodules (see Section~\ref{sec:corep} and Definition~\ref{def:algcoalgfrobalg}). Then there is an isomorphism of categories $$\mathcal{A}^H \cong {\sf Alg}(\mathcal{M}^H).$$ 

\vspace{-.2in}
\qed
\end{theorem}

This result has appeared as a special case of \cite[Proposition~3.9]{BCM}, but we include a full proof here, both for the reader's benefit and to build on it for subsequent results on comodule coalgebras and comodule Frobenius algebras.
The proofs of our results are all constructive, explicit, and written purely in the language of weak bialgebras.

\smallskip

To study comodule coalgebras over weak bialgebras, we note that a formulaic definition of an $H$-comodule coalgebra has appeared in the literature.

\begin{definition}[{\cite[Definition~1]{Jia}, \cite[Definition~2.1]{VZ}}]
Let $H$ be a weak bialgebra and $\varepsilon_s$ be its source counital map [Definition~\ref{def:eps}].
A $\kk$-coalgebra $(C,\;  \Delta:C \to C \otimes C, \; \varepsilon: C \to \kk)$ is called a {\it right $H$-comodule coalgebra} if $C$ is a right $H$-comodule via $\rho: C \to C \otimes H,$ \linebreak $c \mapsto c_{[0]} \otimes c_{[1]}$ so that $c_{1,[0]} \otimes c_{2,[0]} \otimes c_{1,[1]}c_{2,[1]} = c_{[0],1} \otimes c_{[0],2} \otimes c_{[1]}$ and $\varepsilon(c_{[0]})c_{[1]} = \varepsilon(c_{[0]})\varepsilon_s(c_{[1]})$.
\end{definition}

Moreover, we introduce the definition of an $H$-comodule Frobenius algebra.

\begin{definition}[\cref{def:F}] \label{def:wbaFrob-intro}
A Frobenius $\kk$-algebra $(A, m, u, \Delta, \varepsilon)$ is called a {\it right $H$-comodule Frobenius algebra} if $(A,m,u)$ is a right $H$-comodule algebra and $(A, \Delta, \varepsilon)$ is a right $H$-comodule coalgebra.
\end{definition}

We establish results relating the formulaic definitions above with categorical notions.

\begin{theorem}[Theorems~\ref{thm:coalg} and \ref{thm:FrobAlg}] 
\label{thm:coalg-intro}
Let $H$ be a weak bialgebra and let $\mathcal{M}^H$ denote the monoidal category of right $H$-comodules (see Section~\ref{sec:corep}).  Consider the categories below:
\begin{itemize}
    \item the category $\mathcal{C}^H$ of right $H$-comodule coalgebras [\cref{def:C}];
    \item the category $\mathcal{F}^H$ of right $H$-comodule Frobenius algebras [\cref{def:F}];
    \item the category ${\sf Coalg}(\mathcal{M}^H)$ of coalgebras in $\mathcal{M}^H$ [Definition~\ref{def:algcoalgfrobalg}]; and 
    \item the category ${\sf FrobAlg}(\mathcal{M}^H)$ of Frobenius algebras in $\mathcal{M}^H$ [Definition~\ref{def:algcoalgfrobalg}].
\end{itemize} 
Then, we have the following isomorphisms of categories:
\[\mathcal{C}^H \cong {\sf Coalg}(\mathcal{M}^H) \quad \quad \text{and} \quad \quad \mathcal{F}^H \cong {\sf FrobAlg}(\mathcal{M}^H).\]

\vspace{-.25in}
\qed
\end{theorem}


\smallskip

 As mentioned above, our interest in comodule algebras over weak bialgebras is motivated by  coactions of Hayashi's face algebras \cite{H93},
 which include weak bialgebras $\hay$ attached to a finite quiver $Q$ (see Example~\ref{ex:hay}). In \cite[Section~2]{Hayashi99b}, Hayashi showed that the corresponding path algebra $\kk Q$  is a comodule over $\hay$, and we further establish that  $\kk Q$ is an $\hay$-comodule algebra  in order to illustrate Theorem~\ref{thm:alg-intro} [\cref{ex:alg-kQ}]. The  $\hay$-comodule structure on $\kk Q$ is also used for our running examples of weak comodule coalgebras [\cref{ex:coalg-kQ}] and weak comodule Frobenius algebras [\cref{ex:Frobalg-kQ,ex:FrobAlgMatrixAlg}], which illustrate Theorem~\ref{thm:coalg-intro}.

\smallskip

To produce another collection of examples for Theorem~\ref{thm:alg-intro} above, we employ the  quantum transformation groupoids presented in \cite[Section~2.6]{NV}.
These are weak Hopf algebras $H(L,B,\triangleleft)$ that depend on a Hopf algebra $L$, a strongly separable algebra $B$, and an action $B \otimes L \overset{\triangleleft}{\to} B$ making $B$ a right $L$-module algebra [\cref{def:QTG}].
Since a verification that $H(L,B,\triangleleft)$ is indeed a weak Hopf algebra is not provided in the literature, we provide a proof in 
the first preprint version of this work (available at arXiv:math/1911.12847v1). 
Our last  result is that we  construct a monoidal functor from a category of $L$-bicomodules to the corepresentation category $\mathcal{M}^{H(L,B,\triangleleft)}$.

\begin{theorem}[Theorem~\ref{thm:QTG}]
Let $L$, $B$, $\triangleleft$ be as above,  and let $H(L,B,\triangleleft)$ denote the corresponding quantum transformation groupoid. Let $\LBicomod$ be the monoidal category of $L$-bicomodules given in \cref{def:Lbicomod}. Then there is a monoidal functor $$\widehat{\Gamma}: \LBicomod \to \mathcal{M}^{H(L,B,\triangleleft)}.$$ 

\vspace{-.25in}

\qed
\end{theorem}

This theorem provides a mechanism for creating ``weak quantum symmetry" (weak Hopf algebra coactions) from ``ordinary quantum symmetry" (Hopf algebra coactions). 
Several examples of this construction are provided at the end of Section~\ref{sec:QTG}.

\medskip

\noindent {\bf Acknowledgements.}
The authors would like to thank Fabio Calder\'{o}n, Hongdi Huang and James Zhang for helpful discussions that improved the quality of this manuscript.
C. Walton is supported by a research grant from the Alfred P. Sloan foundation and by NSF grants \#DMS-1903192, 2100756. R. Won is supported by an AMS--Simons Travel Grant.


\section{Preliminaries on weak bialgebras} 

In this section, we begin by providing the algebraic set-up material for working over a field in Section~\ref{sec:field}. We then define and recall properties of weak bialgebras and of weak Hopf algebras in Section~\ref{sec:term}. We end by defining (co)modules over weak bialgebras in Section~\ref{sec:modules}. 


\subsection{Algebraic structures over a field} \label{sec:field}
Fix a base field $\kk$, and reserve $\otimes$ to mean $\otimes_{\kk}$. All algebraic structures in this work are $\kk$-linear.

Recall that a {\it $\kk$-algebra} is a $\kk$-vector space $A$ equipped with a multiplication map \linebreak $m: A \otimes A \to A$ and unit map $u: \kk \to A$ so that $m(m \otimes \text{Id}) = m (\text{Id} \otimes m)$ and $m(u \otimes \text{Id}) = \text{Id} = m(\text{Id} \otimes u)$. We reserve the notation 1 to mean
$1:= 1_A:= u(1_\kk).$
A {\it $\kk$-coalgebra} is a $\kk$-vector space $C$ equipped with a comultiplication map $\Delta: C \to C \otimes C$ and counit map $\ep: C \to \kk$ so that $(\Delta \otimes \text{Id})\Delta = (\text{Id} \otimes \Delta)\Delta$ and $(\ep \otimes \text{Id}) \Delta = \text{Id} = (\text{Id} \otimes \ep)\Delta$. When we refer to a (co)algebra, we mean a (co)associative (co)unital (co)algebra.
A {\it Frobenius algebra over $\kk$} is a $\kk$-vector space $A$ that is simultaneously a $\kk$-algebra $(A,m,u)$ and a $\kk$-coalgebra $(A, \Delta, \ep)$ so that $(m \otimes \id)(\id \otimes \Delta)=\Delta m = (\id \otimes m)(\Delta \otimes \id)$.

\smallskip

If $(C, \Delta, \varepsilon)$ is a coalgebra and $c\in C$, we use sumless Sweedler notation:$\Delta(c):=c_1 \otimes c_2,$ 
 $\Delta^2(c)=c_1 \otimes c_2 \otimes c_3$,  etc.
There may be multiple coproducts of $1$ appearing in the same formula, and  we denote different copies of $1$ using primed notation; e.g.,
$\Delta(1)\otimes \Delta(1) = 1_1\otimes 1_2 \otimes 1_1' \otimes 1_2'.$


\subsection{Terminology and properties} \label{sec:term}

\begin{definition} \label{def:wba}
A \textit{weak bialgebra} is a quintuple $(H,m,u,\Delta, \varepsilon)$ such that
\begin{enumerate}[label=(\roman*)]
    \item $(H,m,u)$ is an algebra,
    \item $(H, \Delta, \varepsilon)$ is a  coalgebra,
    \item \label{def:wba3} $\Delta(ab)=\Delta(a)\Delta(b)$ for all $a,b \in H$,
    \item \label{def:wba4} $\varepsilon(abc)=\varepsilon(ab_1)\varepsilon(b_2c)=\varepsilon(ab_2)\varepsilon(b_1c)$ for all $a,b,c \in H$,
    \item \label{def:wba5} $\Delta^2(1)=(\Delta(1) \otimes 1)(1 \otimes \Delta(1))=(1 \otimes \Delta(1))(\Delta(1) \otimes 1)$.
\end{enumerate}
\end{definition}

The difference between a bialgebra and a weak bialgebra can be understood as a weakening of the compatibility between the algebra and coalgebra structures. In a weak bialgebra, we still have that comultiplication is multiplicative (e.g.,  condition \ref{def:wba3}), but the counit is no longer multiplicative and we do not necessarily have $\Delta(1)=1 \otimes 1$ or $\varepsilon(1)=1$. Instead, we have weak multiplicativity of the counit (condition \ref{def:wba4}) and weak comultiplicativity of the unit (condition \ref{def:wba5}). 

\begin{definition}[$\varepsilon_s$, $\varepsilon_t$, $H_s$, $H_t$] \label{def:eps}
Let $(H, m, u, \Delta, \varepsilon)$ be a weak bialgebra. We define the {\it source and target counital maps}, respectively as follows:
\[
    \varepsilon_s: H \to H, \; x \mapsto 1_1\ep(x1_2) \qquad \text{and} \qquad
    \ept: H \to H, \;  x \mapsto \ep(1_1x)1_2.
\]
We denote the images of these maps as $H_s:=\eps(H)$ and $H_t:=\ept(H)$. We call $H_s$ and $H_t$ the {\it counital subalgebras} of $H$ (see \cref{prop:Hs-facts}).
\end{definition}

With these definitions in mind, we can define a weak Hopf algebra.

\begin{definition} \label{def:wha}
A \textit{weak Hopf algebra} is a sextuple $(H,m,u,\Delta,\varepsilon, S)$, where $(H,m,u,\Delta,\varepsilon)$ is a weak bialgebra and $S: H \to H$ is a $\kk$-linear map called the \textit{antipode} that satisfies the following properties for all $h\in H$:
    $S(h_1)h_2=\varepsilon_s(h)$,
    $h_1S(h_2)=\varepsilon_t(h)$,
    $S(h_1)h_2S(h_3)=S(h)$.
\end{definition}


We notice that the antipode axioms for a weak Hopf algebra use the maps $\eps, \ept$ instead of $\ep$ as for ordinary Hopf algebras. In fact, the counital subalgebras $H_s=\eps(H)$ and $H_t=\ept(H)$ are often used instead of $\kk = \ep(H)$ in the weak Hopf algebra setting. These subalgebras have special properties that we will need below.

\begin{proposition}[Facts about counital subalgebras] \label{prop:Hs-facts} 

Let $H$ be a weak bialgebra.
\begin{enumerate}[label=(\roman*),font=\upshape]
    \item\label{prop:Hs-facts-1} \cite[Prop. 2.2.1(iii)]{NV}  $h \in H_s \Leftrightarrow \Delta(h)=1_1 \otimes h 1_2$ and $h \in H_t \Leftrightarrow \Delta(h)=1_1h \otimes 1_2$.
    \item \cite[Prop. 2.2.2]{NV} $H_s$ (resp., $H_t$) is a left (resp., right) coideal subalgebra of $H$.
    \item \cite[Prop. 2.2.2]{NV} If $y \in H_s$ and $z \in H_t$, then $yz=zy$. 
    \item \cite[Prop. 1.17]{BCJ} \label{prop:Hs-facts-5}  $H_s$ is a coalgebra with counit $\varepsilon|_{H_s}$ and comultiplication  
$$\Delta_s: H_s \to H_s \otimes H_s, \quad
y \mapsto 1_1 \otimes \varepsilon_s(y 1_2)=y_1 \otimes \eps(y_2). $$
    \item \cite[Prop. 1.17]{BCJ}  $H_t$ is a coalgebra with counit $\varepsilon|_{H_t}$ and comultiplication  
$$\Delta_t: H_t \to H_t \otimes H_t, \quad
z \mapsto \ept(1_1z) \otimes 1_2 =\ept(z_1)\otimes z_2.$$ 
\end{enumerate}
\end{proposition}

\begin{proposition}[Useful weak bialgebra identities] \cite[Prop. 2.2.1]{NV}
Let $H$ be a weak bialgebra and let $g,h \in H$. The following relations hold:
\begin{align}
    &\label{eq:a}\eps(\eps(h))=\eps(h), \hspace{1.05in} \ept(\ept(h))=\ept(h), \\
    & \label{eq:g} h_1\eps(h_2)=h=\ept(h_1)h_2, 
    \\
    & \label{eq:Delta1} \Delta(1) = 1_1 \otimes \ept(1_2) = \eps(1_1) \otimes 1_2 \in H_s \otimes H_t. 
\end{align}
\end{proposition}

 \begin{example}[Hayashi's face algebras attached to quivers]
\label{ex:hay}
Let $Q$ be a finite quiver. Denote by $Q_\ell$ the set of all paths of length $\ell$ in $Q$, and let $Q_0 = \{1, 2, \dots, n\}$. For a path $p$, let $s(p)$ and $t(p)$ denote the source and target of $p$, respectively. We define the weak bialgebra $\mfH(Q)$ as follows. As a $\kk$-vector space, $\mfH(Q) = \bigoplus_{\ell \geq 0} \bigoplus_{p, q \in Q_{\ell}} \kk x_{p,q}$, for indeterminates $x_{p,q}$. The algebra structure of $\mfH(Q)$ is given by
\[ x_{p,q} \cdot x_{p',q'} 
= \delta_{t(p), s(p')} \delta_{t(q), s(q')} x_{pp',qq'} 
=\begin{cases} x_{pp',qq'} & \text{if $t(p) = s(p')$ and $t(q) = s(q')$} \\
0 & \text{otherwise}
\end{cases}
\]
with $1 = \sum_{i, j \in Q_0} x_{i,j}$. For $p,q \in Q_{\ell}$, the coalgebra structure is given by
\[ \Delta\left(x_{p,q}\right) = \textstyle \sum_{t \in Q_{\ell}} x_{p,t} \otimes x_{t,q}
\quad \quad \text{and} \quad \quad \varepsilon\left(x_{p,q}\right) = \delta_{p,q}.
\]

For each $j \in Q_0$, the {\it face idempotents} of $\mfH(Q)$ are  $a_j:=\sum_{i=1}^n x_{i,j}$ and $a'_j:=\sum_{i=1}^n x_{j,i}$. For $p,q \in Q_{\ell}$, we can compute $\eps(x_{p,q}) = \delta_{p,q} \sum_{i = 1}^n x_{i, t(q)}$ and $\ept(x_{p,q}) = \delta_{p,q} \sum_{j = 1}^n x_{s(q),j}$. Hence, as $\kk$-vector spaces, $\mfH(Q)_s = \bigoplus_{j =1}^n \kk a_j$ and $\mfH(Q)_t = \bigoplus_{j=1}^n \kk a'_j$.
We refer to the reader to \cite{H93} for more details. 
\end{example}


\subsection{Modules and comodules} \label{sec:modules}

We now direct our attention to the (co)representation theory of weak bialgebras.
Let $H$ be a weak bialgebra and let $M$ be a $\kk$-vector space. We call $M$ a right \textit{$H$-(co)module} if it is a right (co)module under the (co)algebra structure  of $H$. We define left \textit{$H$-(co)modules} similarly. We call $M$ an \textit{$(H,H)$-bi(co)module} if it is a bi(co)module under the (co)algebra structure of $H$.

 Given an algebra $A$, we denote the actions for a right $A$-module $M$ and for a left $A$-module $N$ as follows:
\[\nu_M: M \otimes A \to M, \; \; m \otimes a \mapsto m \triangleleft a,\]
\[\mu_N: A \otimes N \to N, \;\; \; a \otimes n \mapsto a \triangleright n.\]

Now suppose that $M$ and $N$ are $(A,A)$-bimodules. Recall that the {\it tensor product} of $M$ and $N$ is an $(A,A)$-bimodule given as follows:
\[M \otimes_A N := (M \otimes N)/ \;\text{im}(\nu_M \otimes \id_N - \id_M \otimes \mu_N).\]
That is, $M \otimes_A N$ is a $\kk$-quotient space of $M \otimes N.$

 Given a coalgebra $C$, we denote the coactions for a right $C$-comodule $M$ and for a left $C$-comodule $N$ using sumless Sweedler notation:
\[\rho_M: M \to M \otimes C, \; \; m \mapsto m_{[0]}\otimes m_{[1]}, \quad\quad \lambda_N: N \to C \otimes N, \; \; n \mapsto n_{[-1]}\otimes n_{[0]}.\]
A (right) $C$-comodule $M$ must satisfy the coassociativity and unitality axioms: 
\begin{eqnarray}
    m_{[0],[0]} \otimes m_{[0],[1]} \otimes m_{[1]} &=& m_{[0]} \otimes m_{[1],1} \otimes m_{[1], 2}, \label{eq:Hcomod-1}\\
    m &=& \varepsilon(m_{[1]}) m_{[0]}. \label{eq:Hcomod-2}
\end{eqnarray}
for all $m\in M$. We let $\ComodC$ denote the category of right $C$-comodules. 

Now let $M,N$ be $(C,C)$-bicomodules. 
Recall that the {\it cotensor product} of $M$ and $N$ is a $(C,C)$-bicomodule given as follows:
\[M \otimes^C N := \ker \l(\rho_M \otimes \id_N - \id_M \otimes \lambda_N\r).\]
By definition, $M \otimes^C N$ is a $\kk$-subspace of $M \otimes N.$

\section{Monoidal categories and corepresentation categories} 

In this section, we begin by recalling facts about monoidal categories, functors between them, and algebraic structures within them in Section~\ref{sec:monoidal}. We then recall the definition of the corepresentation category of a weak bialgebra in Section~\ref{sec:corep}, and discuss functors between this category and other monoidal categories in Section~\ref{sec:MHfunctors}.

\subsection{Monoidal categories} \label{sec:monoidal}

\begin{definition}
A \textit{monoidal category} $\mathcal{C}$ consists of the following data:
\begin{itemize}
    \item a category $\CC$,
    \item a bifunctor $\otimes : \CC\times \CC\rightarrow \CC$,
    \item a natural isomorphism   $\alpha_{X,Y,Z}: (X \otimes Y)\otimes Z \overset{\sim}{\to} X \otimes (Y\otimes Z)$ for each $X,Y,Z \in \CC$, 
    \item an object $\unit \in \CC$,
    \item natural isomorphisms $l_X: \mathbbm{1} \otimes X \overset{\sim}{\to} X, \:\: r_X:X \otimes \mathbbm{1} \overset{\sim}{\to} X$ for each $X \in \CC$,
\end{itemize}
 such that the pentagon and triangle axioms are satisfied (see \cite[Eq. 2.2, 2.10]{EGNO}).
\end{definition}

\begin{example}\label{ex:monoidcat}
Examples of monoidal categories include the following:

\begin{itemize}
    \item $\Vec$, the category of finite-dimensional $\kk$-vector spaces, with $\otimes = \otimes_\kk$, $\unit = \kk$, and with the canonical associativity and unit isomorphisms;
    \smallskip
    \item ${}_A \mc{M}_A$,  the category of $(A,A)$-bimodules with monoidal product being the tensor product $\otimes_A$, for an algebra $A$;
    \smallskip
    \item ${}^C \mc{M}^C$,  the category of $(C,C)$-bicomodules with monoidal product being the cotensor product $\otimes^C$, for a coalgebra $C$. 
\end{itemize}

 If $C$ is a weak bialgebra, we can endow the category of right $C$-comodules with the structure of a monoidal category; see \cref{sec:corep} below. 
 \end{example}

\begin{definition} \label{def:monfunctor} \cite[p.85]{Str} \cite{DP08} \cite[Eq.6.46, 6.47]{Sz} Let $(\CC, \otimes_{\CC}, \unit_{\CC}, \alpha_{\ast,\ast,\ast}, l_\ast, r_\ast)$ and  $(\DD,\otimes_{\DD},\unit_{\DD},\alpha_{\ast,\ast,\ast}, l_\ast, r_\ast)$ be monoidal categories.

\begin{enumerate}[label=(\roman*)]
    \item  A \textit{monoidal functor} $(F, F_{\ast,\ast}, F_0): \CC \to \DD$ consists of the following data:
    \begin{itemize}
        \item a functor $F: \CC \to \DD$,
        \item a natural transformation $F_{X,Y}:F(X)\otimes_{\DD} F(Y) \to F(X \otimes_{\CC} Y)$ for all $X,Y \in \CC$,
        \item a morphism $F_0: \unit_{\DD} \to F(\unit_{\CC})$ in $\DD$,
    \end{itemize}
    that satisfy the associativity and unitality constraints: for all $X,Y,Z \in \mathcal{C}$, 
\[
    \begin{array}{ll}
    \smallskip
    &F_{X,Y\otimes_\mathcal{C} Z}\;(\id_{F(X)} \otimes_{\mathcal{D}} F_{Y,Z})\;\alpha_{F(X),F(Y),F(Z)}\\
    &\hspace{1.3in} = F(\alpha_{X,Y,Z})\;F_{X \otimes_\mathcal{C} Y,Z}\;(F_{X,Y}\otimes_{\mathcal{D}}\id_{F(Z)}),
    \end{array}
  \]
    
    \vspace{-.04in}
    
    \[
    \begin{array}{rl}
    \smallskip
    F(l_X)^{-1}\;l_{F(X)} &= F_{\unit_\mathcal{C},X} \; (F_0 \otimes_{\mathcal{D}} \id_{F(X)}),\\
    F(r_X)^{-1}\;r_{F(X)} &= F_{X,\unit_\mathcal{C}} \; (\id_{F(X)} \otimes_{\mathcal{D}} F_0).
    \end{array}
    \]
A \textit{strong monoidal functor} is a monoidal functor where $F_0,F_{X,Y}$ are isomorphisms for all $X,Y \in \CC$. 
    \smallskip
    \item  A \textit{comonoidal functor} $(F, F^{\ast,\ast}, F^0): \CC \to \DD$ consists of the following data:
    \begin{itemize}
        \item a functor $F: \CC \to \DD$,
        \item a natural transformation $F^{X,Y}: F(X \otimes_{\CC} Y) \to F(X)\otimes_{\DD} F(Y)$ for all $X,Y \in \CC$,
        \item a morphism $F^0: F(\unit_{\CC}) \to  \unit_{\DD}$ in $\DD$,    
    \end{itemize}
    that satisfy the coassociativity and counitality constraints dual to above. A \textit{strong comonoidal functor} is a comonoidal functor where $F^0,F^{X,Y}$ are isomorphisms for all $X,Y \in \CC$. 
    \smallskip
    \item  A \textit{Frobenius monoidal functor} $(F, F_{\ast,\ast}, F_0, F^{\ast,\ast}, F^0): \CC \to \DD$ is a functor  where $(F, F_{\ast,\ast}, F_0)$ is monoidal, $(F, F^{\ast,\ast}, F^0)$ is comonoidal, such that for all $X,Y,Z \in \CC$:
    \begin{align*}
   \textcolor{white}{.......} \; (F_{X,Y} \otimes_{\DD} \id_{F(Z)} )\alpha^{-1}_{F(X),F(Y),F(Z)}(\id_{F(X)} \otimes_{\DD} F^{Y,Z}) &= F^{X\otimes_{\CC}Y,Z}F(\alpha^{-1}_{X,Y,Z}) F_{X,Y\otimes_{\CC}Z}, \\
    (\id_{F(X)} \otimes_{\DD} F_{Y,Z}) \alpha_{F(X),F(Y),F(Z)} (F^{X,Y}\otimes_{\DD} \id_{F(Z)}) &= F^{X,Y\otimes_{\CC}Z}F(\alpha_{X,Y,Z})F_{X\otimes_{\CC}Y,Z}.
    \end{align*}
\end{enumerate}
\end{definition}

It is well known that a strong monoidal functor is the same as a strong comonoidal functor.  Moreover, any strong (co)monoidal functor is Frobenius monoidal  \cite[Proposition~3]{DP08}.

Now we turn our attention to algebraic structures in monoidal categories.

\begin{definition}[${\sf Alg}(\CC)$, ${\sf Coalg}(\CC)$, ${\sf FrobAlg}(\CC)$] 
\label{def:algcoalgfrobalg} Let $(\CC, \otimes, \unit, \alpha, l, r)$ be a  monoidal category.
\begin{enumerate}[label=(\roman*)]
    \item An \textit{algebra} in $\CC$ is a triple $(A,m,u)$, where $A$ is an object in $\CC$, and  $m:A \otimes A \to A$, $u:\unit \to A$ are morphisms in $\CC$,
    satisfying unitality and associativity constraints: 
    \begin{align} \label{eq:algassoc}
       m (m \otimes \id) = m(\id \otimes m) \alpha_{A,A,A},\quad \quad\\
       m (u \otimes \id) = l_A, \quad \quad  \quad m(\id \otimes u) = r_A. \label{eq:algunital}
    \end{align}
    A {\it morphism} of algebras $(A, m_A, u_A)$ to $(B, m_{B}, u_{B})$  is a morphism $f: A \to B$ in $\CC$ so that $fm_A = m_{B \otimes B}(f \otimes f)$ and $fu_A = u_{B}$.  Algebras in $\CC$ and their morphisms  form a category, which we denote by ${\sf Alg}(\CC)$.
    \smallskip
    
    \item A \textit{coalgebra} in $\CC$ is a triple $(C,\Delta,\varepsilon)$, where  $C$ is an object in $\CC$, and $\Delta:C \to C \otimes C$,  $\varepsilon:C \to \unit$ are morphisms in $\CC$,
    satisfying counitality and coassociativity constraints: 
    \begin{align}
        \label{eq:coalgcoassoc}
        \alpha_{C,C,C}(\Delta \otimes \id) \Delta = (\id \otimes \Delta)\Delta, \quad \quad \\ \label{eq:coalgcounital}
        (\ep \otimes \id)\Delta = l_C^{-1}, \quad \quad \quad (\id \otimes \ep)\Delta = r_C^{-1}.
    \end{align}
  A {\it morphism} of coalgebras $(C, \Delta_C, \ep_C)$ to $(D, \Delta_{D}, \ep_{D})$  is a morphism $g: C \to D$ in $\CC$ so that $\Delta_{D}g  = (g \otimes g) \Delta_{C}$ and $\ep_{D} g = \ep_C$. Coalgebras in $\CC$ and their morphisms  form a category, which we denote by ${\sf Coalg}(\CC)$.
  
   \smallskip
   
    \item A \textit{Frobenius algebra} in $\CC$ is a quintuple $(A, m, u, \Delta, \varepsilon)$ such that
   $(A,m,u) \in {\sf Alg}(\CC)$ and $(A,\Delta,\varepsilon) \in {\sf Coalg}(\CC)$, so that 
       \begin{equation}
       \label{eq:Frobalg}
       (m \otimes \id)\alpha^{-1}_{A,A,A}(\id \otimes \Delta)=\Delta m = (\id \otimes m)\alpha_{A,A,A}(\Delta \otimes \id).\end{equation}
    A {\it morphism} of Frobenius algebras in $\CC$ is a morphism in $\CC$ that lies in both ${\sf Alg}(\CC)$ and ${\sf Coalg}(\CC)$. 
    Frobenius algebras in $\CC$ and their morphisms  form a category, which we denote by ${\sf FrobAlg}(\CC)$.
\end{enumerate}
\end{definition}

Algebras, coalgebras, and Frobenius algebras in $\Vec$ are the same as $\kk$-algebras, $\kk$-coalgebras, and Frobenius algebras over $\kk$, respectively (cf. Section~\ref{sec:field}).
Next, we see that monoidal, comonoidal, and Frobenius monoidal  functors preserve, respectively, algebras, coalgebras, and Frobenius algebras in monoidal categories. 

\begin{proposition} \label{prp:monfunctor}
\cite[p.100-101]{Str} \cite[Lem. 2.1]{Sz} \cite[Cor. 5]{DP08} \cite[Prop.~2.13]{KongRunkel}
\begin{enumerate}[label=(\roman*), font=\upshape]
    \item \label{prp:laxmon}  Let $(F, F_{\ast,\ast}, F_0): \CC \to \DD$ be a monoidal functor. If $(A,m,u) \in {\sf Alg}(\CC)$, then $$(F(A),~ F(m)F_{A,A}, ~F(u)F_0) \in {\sf Alg}(\DD).$$
    \item \label{prp:laxcomon}  Let $(F, F^{\ast,\ast},F^0): \CC \to \DD$ be a comonoidal functor. If $(C,\Delta,\varepsilon) \in {\sf Coalg}(\CC)$, then 
    $$(F(C),~ F^{C,C}F(\Delta), ~F^0F(\varepsilon)) \in {\sf Coalg}(\DD).$$ 
    \item \label{prp:frobmon}  Let $(F, F_{\ast,\ast},F_0,F^{\ast,\ast},F^0):\CC \to \DD$ be a Frobenius monoidal functor. If $(A,m,u,\Delta,\varepsilon)$ is in ${\sf FrobAlg}(\CC)$, then 
    \[
    (F(A),\; F(m)F_{A,A}, \; F(u)F_0, \; F^{A,A}F(\Delta),\;  F^0F(\varepsilon)) \in {\sf FrobAlg}(\DD).
    \]
    
    \vspace{-.25in}
    
    \qed
\end{enumerate}
\end{proposition}


\subsection{\texorpdfstring{The corepresentation category $\MH$}{The corepresentation category}} \label{sec:corep}
For a weak bialgebra $H = (H, m, u, \Delta, \varepsilon)$, the category $\MH$ of right $H$-comodules can be given the structure of a monoidal category:
\[\MH = ({\sf Comod}\text{-}H,  \hspace{.05in} \tenbar, \hspace{.05in} \unit = H_s, \hspace{.05in} \alpha = \alpha_{{\sf Vec}_\Bbbk}, \hspace{.05in} l, \hspace{.05in} r).
\]
We give further details about the tensor product $\tenbar$ and unitality constraints $l$ and $r$ below. We also recall some basic facts about $\MH$ which were provided in \cite{BCJ}, some of which appeared in the preprint \cite{Nill}. The connection between $\MH$ and the bi(co)module categories ${}_{H_s} \mc{M}_{H_s}$ and ${}^{H_s} \mc{M}^{H_s}$ are also highlighted below.

\begin{remark}
By \cite[Remark~2.4.1]{NV}, $H^{cop}$ is a weak Hopf algebra. Hence the category of left $H$-comodules is isomorphic to the category of right $H^{cop}$-comodules and we can identify ${}^H\mc{M}$ with $\mc{M}^{H^{cop}}$. Note that the category of left $H$-comodules can also be given the structure of a monoidal category, but one would need to use the counital subalgebra $H_t$ in place of $H_s$, along with implementing other similar substitutions.
\end{remark}
 
\subsubsection{\texorpdfstring{Tensor product $\tenbar$}{Tensor product}} \label{sec:otimes-MH} Given $M,N \in \MH$, the tensor product $M \otimes N$ has a coassociative, but not necessarily counital, right $H$-coaction given by $m \otimes n \mapsto m_{[0]} \otimes n_{[0]} \otimes m_{[1]}n_{[1]}$. Hence, in $\MH$ we force counitality by defining  the monoidal product of $M$ and $N$  to be
\begin{equation} \label{eq:tenbar}
 M \tenbar N := \left\{m \otimes n \in M \otimes N \mid m \otimes n = \varepsilon(m_{[1]}n_{[1]}) m_{[0]} \otimes n_{[0]} \right\}.
\end{equation}

We use the notation $m \tenbar n$ for the element $m \otimes n$ when viewed as an element of $M \tenbar N$. The right $H$-coaction on $M \tenbar N$ is given by
\begin{equation} \label{eq:rhoMN}
\begin{gathered}
\rho_{M  \tenbar  N}: M \tenbar N \to M \tenbar N \otimes H \\
\hspace{1.1in} m  \tenbar n \mapsto m_{[0]} \tenbar n_{[0]} \otimes m_{[1]}n_{[1]}.
 \end{gathered}
\end{equation}
The inclusion of $M \tenbar N$ into $M \otimes N$ is given by 
\begin{equation*}
\begin{gathered}
\hspace{-.22in}\iota_{M,N}: M \tenbar N \to M \otimes N \\
\hspace{.2in} \varepsilon(m_{[1]}n_{[1]}) m_{[0]}  \tenbar n_{[0]} \mapsto  \varepsilon(m_{[1]}n_{[1]}) m_{[0]} \otimes n_{[0]}.
 \end{gathered}
\end{equation*}
There is also a projection
\begin{equation*} 
\begin{gathered}
\eta_{M,N}: M \otimes N  \to M \tenbar N \\
\hspace{1.1in} m \otimes n \mapsto \varepsilon(m_{[1]}n_{[1]}) m_{[0]}  \tenbar n_{[0]}.
\end{gathered}
\end{equation*}

Next, we study $\tenbar$ of morphisms of $H$-comodules. 
If $f: M \to M'$ and $g: N \to N'$ are morphisms of right $H$-comodules, then we define 
\begin{align*} f \tenbar g: M \tenbar N &\to M' \tenbar N' \\
\ep(m_{[1]}n_{[1]})m_{[0]} \tenbar n_{[0]} &\mapsto \ep(m_{[1]}n_{[1]})f(m_{[0]}) \tenbar g(n_{[0]}).
\end{align*}
Note that since $f$ and $g$ are right $H$-comodule morphisms, we have 
\[
    \ep(f(m)_{[1]} g(n)_{[1]}) f(m)_{[0]} \tenbar g(n)_{[0]} = \ep(m_{[1]}n_{[1]})f(m_{[0]}) \tenbar g(n_{[0]}), 
\]
and so $\ep(m_{[1]}n_{[1]})f(m_{[0]}) \tenbar g(n_{[0]}) \in M' \tenbar N'$. This shows that for $m \tenbar n \in M \tenbar N$, \begin{equation}\label{eq:f-tenbar-g}
(f \tenbar g)(m \tenbar n) = f(m) \tenbar g(n).
\end{equation}

\subsubsection{\texorpdfstring{Unit object $~\unit = 
H_s$}{Unit object}} 
\label{sec:unitobj-MH}
The counital subalgebra $H_s$ is naturally a right $H$-comodule since the image of $\Delta|_{H_s}$ is a subspace of $H_s \otimes H$ (see Proposition~\ref{prop:Hs-facts}(iv)), and so $\Delta|_{H_s}$ can be viewed as a map $H_s \to H_s \otimes H$. By \cite[Theorem 3.1]{BCJ}, $H_s$ is the unit object of the monoidal category $\MH$.

\subsubsection{\texorpdfstring{The unit isomorphisms $l$ and $r$}{The unit isomorphisms}} \label{sec:unit-MH} By \cite[Section~3]{BCJ}, the monoidal category $\MH$ has unit isomorphisms:
$$l_M: H_s \tenbar M \to M, \quad \quad x \; \bar{\otimes} \; m = \varepsilon(x_{[1]}m_{[1]}) x_{[0]} \tenbar m_{[0]} ~~\mapsto~~ \varepsilon(x m_{[1]}) m_{[0]},$$

\vspace{-.2in}

$$r_M:  M \tenbar H_s \to M, \quad \quad m \; \bar{\otimes} \; x = \varepsilon(m_{[1]}x_{[1]}) m_{[0]} \tenbar x_{[0]} ~~\mapsto~~ \varepsilon(m_{[1]} x) m_{[0]},$$
for all $M \in \MH$. Moreover, the inverses of these maps are given as follows:
\[l_M^{-1}: M \to  H_s \tenbar M, \quad \quad m  ~~\mapsto~~  \varepsilon(1_2 m_{[1]})1_1  \tenbar  m_{[0]},\]

\vspace{-.2in}

\[r_M^{-1}: M \to  M \tenbar H_s, \quad \quad m  ~~\mapsto~~  m_{[0]} \tenbar \varepsilon(m_{[1]} 1_2)1_1.
\]

\subsection{\texorpdfstring{Functors between $\MH$ and other monoidal categories}{Functors between the corepresentation category and other monoidal categories}} \label{sec:MHfunctors}

In this section, we discuss the functors $U: \MH \to \Vec$, $\Phi:  \MH \to {}_{H_s} \mc{M}_{H_s}$, and $\Psi:  \MH \to {}^{H_s} \mc{M}^{H_s}$. We will see below that $U$ and $\Phi$ are monoidal, and also that $U$ and $\Psi$ are comonoidal.

\subsubsection{\texorpdfstring{The forgetful functor $U: \MH \to \Vec$}{The forgetful functor U}} \label{sec:U} To begin, we have by \cite[Theorem~3.2]{BCJ} that the forgetful functor $$U : \MH \to \Vec$$
is both a monoidal functor via
\begin{nalign} 
\label{eq:Ulower} U_{M,N}: M \otimes N &\overset{\eta_{M,N}}{\xrightarrow{\hspace*{1cm}}} M \tenbar N  &   &\text{and} &  U_0: \kk &\overset{u_{H_s}}{\xrightarrow{\hspace*{.7cm}}} H_s \\
m \otimes n &\mapsto \varepsilon(m_{[1]}n_{[1]}) m_{[0]}  \tenbar n_{[0]} & & & 1_{\kk} &\mapsto 1_H,
\end{nalign}
and a comonoidal functor via
\begin{nalign} 
\label{eq:Uupper} U^{M,N}: M \tenbar N &\overset{\iota_{M,N}}{\xrightarrow{\hspace*{1cm}}} M \otimes N & &\text{and}  &  U^0: H_s &\overset{\varepsilon|_{H_s}}{\xrightarrow{\hspace*{.7cm}}} \kk  \\
 \varepsilon(m_{[1]}n_{[1]}) m_{[0]}  \tenbar n_{[0]} &\mapsto  \varepsilon(m_{[1]}n_{[1]}) m_{[0]} \otimes n_{[0]} & & & x &\mapsto \ep(x).
\end{nalign}
Here, $\eta_{M,N}$ and $\iota_{M,N}$ are defined in Section~\ref{sec:otimes-MH}. 
In fact, by \cite[Section 6]{Sz}, $U$ is a Frobenius monoidal functor.
We also remark that 
\begin{equation}
\label{eq:UMNidentity}
    U_{M,N}U^{M,N} = \id_{M \bar{\otimes}N}.
\end{equation}

\subsubsection{\texorpdfstring{The functor $\Phi:  \MH \to {}_{H_s} \mc{M}_{H_s}$}{The functor Phi}} \label{sec:Phi} The monoidal structure of $U$ induces an $H_s$-bimodule action on any right $H$-comodule. Explicitly, if $M$ is a right $H$-comodule, then the left $H_s$-action $\mu_M$ and right $H_s$-action $\nu_M$ are given by, respectively,  
\begin{equation} \label{eq:Hs-act}
   \mu_M(x \otimes m)=:  \;x \triangleright m = \varepsilon(xm_{[1]})m_{[0]}, \quad \quad \nu_M(m \otimes x)=: \; m \triangleleft x = \varepsilon(m_{[1]}x)m_{[0]}, 
\end{equation}
for all $x \in H_s$, $m \in M$; this is consistent with Section~\ref{sec:unit-MH}. Moreover, these actions commute with right $H$-coaction in the following way: for all $x \in H_s, m \in M$,
\begin{equation} \label{eq:Hs-act-Hcom}
    (x \triangleright m)_{[0]} \otimes (x \triangleright m)_{[1]} = m_{[0]} \otimes x m_{[1]} \quad \text{and} \quad (m \triangleleft x)_{[0]} \otimes (m \triangleleft x)_{[1]} = m_{[0]} \otimes  m_{[1]}x.
\end{equation}
For a morphism $f:M\to N \in \MH$, define $\Phi(f) = f$. 
Indeed, $f$ is also an $H_s$-bimodule morphism as one can check that $f(x \triangleleft m) = x \triangleleft f(m)$ and $f(x \triangleright m) = x \triangleright f(m)$.

Next, note that as a monoidal functor, $U$ factors through the monoidal functor 
\[\Phi: \MH \to {}_{H_s} \mc{M}_{H_s}\quad \text{where} \quad
\begin{aligned} 
\Phi_{M,N}: M \otimes_{H_s} N &\to M \tenbar N\\
m \otimes_{H_s} n &\mapsto \varepsilon(m_{[1]}n_{[1]}) m_{[0]}\tenbar n_{[0]}.
\end{aligned}
\]
and $\Phi_0  = \id_{H_s}$. In fact, $\Phi_{M,N}$ is an isomorphism of $H$-comodules, with inverse
\begin{align*} 
\hspace{1in} \Phi_{M,N}^{-1}: M \tenbar N &\to M \otimes_{H_s} N\\
\varepsilon(m_{[1]}n_{[1]}) m_{[0]} \tenbar n_{[0]} &\mapsto \varepsilon(m_{[1]}n_{[1]}) m_{[0]} \otimes_{H_s} n_{[0]} = m \otimes_{H_s} n.
\end{align*}
Throughout, we use $\Phi_{M,N}$ and $\Phi_{M,N}\inv$ to identify $M \tenbar N$ with $M \otimes_{H_s} N$.

\subsubsection{\texorpdfstring{The functor $\Psi:  \MH \to {}^{H_s} \mc{M}^{H_s}$}{The functor Psi}} \label{sec:Psi} The comonoidal structure of $U$ induces an $H_s$-bicomodule action on any right $H$-comodule. Explicitly, if $M$ is a right $H$-comodule then the left $H_s$-coaction and right $H_s$-coaction on $M$ are given respectively by 
\begin{equation} \label{eq:Hs-coact}
    \lambda_{M}^s(m):= \varepsilon\big(m_{[1]}1_{1}\big)1_{2} \otimes m_{[0]},  \quad \rho_{M}^s(m):= m_{[0]} \otimes 1_{1}\varepsilon\big( m_{[1]}1_{2}\big) \quad \quad \forall m \in M.
\end{equation}
    As a comonoidal functor, $U$ factors through the comonoidal functor 
    \[
    \Psi: \MH \to {}^{H_s} \mc{M}^{H_s}\quad \text{where} \quad
\begin{aligned}  \Psi^{M,N}: M \tenbar N &\to M \otimes^{H_s} N\\ 
\varepsilon(m_{[1]}n_{[1]}) m_{[0]} \otimes n_{[0]} &\mapsto \varepsilon(m_{[1]}n_{[1]}) m_{[0]} \otimes^{H_s} n_{[0]}
\end{aligned}
\]
is the inclusion and $\Psi^0: H_s \to H_s$ is the identity map. Here, $M \otimes^{H_s} N$ is the cotensor product of $M$ and $N$ using the $H_s$-coactions given in \eqref{eq:Hs-coact} and \cref{def:eps}. In fact, $M \tenbar N$ and $M \otimes^{H_s} N$ are equal as subspaces of $M \otimes N$.


\section{Weak comodule algebras} \label{sec:algebras}

The goal of this section is to reconcile the two notions of comodule algebras over weak bialgebras available in the literature, which is a special case of \cite[Proposition~3.9]{BCM}.

\smallskip

For the remainder of the paper, fix $H = (H,m,u,\Delta,\varepsilon)$ a weak bialgebra, and set the notation $1:=1_H = u(1_\Bbbk).$

\begin{definition}[$\mathcal{A}^H =:\mathcal{A}$] \label{def:A} and \cite[Definition~2.1]{Bohm} \cite[Proposition~4.10]{C-DG}. Consider the {\it category $\mathcal{A}^H$ of right $H$-comodule algebras} defined as follows. The objects of $\mathcal{A}^H$ are algebras in~${\sf Vec}_\Bbbk$, 
$(A,~m_A: A \otimes A \to A, ~u_A: \Bbbk \to A),$
with $1_A:=u_A(1_\Bbbk)$, so that 
\begin{itemize}
\item the $\Bbbk$-vector space $A$ is a right $H$-comodule via $\rho_A: A \to A \otimes H, \; a \mapsto a_{[0]} \otimes a_{[1]},$ namely, the coassociativity and counit axioms hold (equations~\eqref{eq:Hcomod-1} and \eqref{eq:Hcomod-2});
\item the multiplication map $m_A$ is compatible with $\rho_A$ in the sense that 
\begin{eqnarray} 
\hspace{1.5in}    (ab)_{[0]} \otimes (ab)_{[1]} &=& a_{[0]} b_{[0]} \otimes a_{[1]} b_{[1]} \hspace{1.1in}  \forall a,b \in A; \label{eq:m-A-2}
\end{eqnarray}
\item the unit map $u_A$ is compatible with $\rho_A$ in the sense that  
\begin{eqnarray}
\rho_A(1_A) &\in& A \otimes H_t. \label{eq:u-A-1}
\end{eqnarray}
\end{itemize}
The morphisms of $\mathcal{A}^H$ are $\Bbbk$-linear algebra maps  that are also $H$-comodule maps. We write $\mathcal{A}$ for this category when $H$ is understood.
\end{definition}

Recall from Section~\ref{sec:corep}  the monoidal structure of the category $\MH$ of right $H$-comodules. Our main result is the following.

\begin{theorem} \label{thm:alg}
The category $\mathcal{A}^H=:\mathcal{A}$ of right $H$-comodule algebras is isomorphic to the category {\sf Alg}$(\MH)$ of algebras in the monoidal category $\MH$. 
\end{theorem}

\begin{proof}
We achieve this result by constructing mutually inverse functors $F$ and $G$ below.

\vspace{-.1in}

\[
\begin{tikzcd}
\mathcal{A} \arrow[shift left]{rr}{F} & & {\sf Alg}(\MH) \arrow[shift left]{ll}{G}
\end{tikzcd}
\]

First, take $(A, ~m_A, ~u_A) \in \mathcal{A}$, and consider the assignment 
\[
\begin{array}{rl}
\medskip
 F(A,~ m_A, ~u_A) &:=
(A, \quad m_A \; U^{A,A}: A \tenbar A \to A, \quad \nu_A  (u_A \otimes \id_{H_s}): H_s \to A)\\
F(f) &:= f
\end{array}
\]
 for $f: (A, m_A, u_A) \to (B, m_B, u_B) \in \mathcal{A}$. 
 Here,  $U^{A,A}$, $\nu_A$ are given in \eqref{eq:Uupper}, \eqref{eq:Hs-act}, resp.

\begin{claim} \label{claim:Alg.F}
The assignment $F$ above yields a functor from $\mathcal{A}$ to ${\sf Alg}(\mathcal{M}^{H})$.
\end{claim}

\begin{proof}[Proof of Claim~\ref{claim:Alg.F}]
First, $A$ is a right $H$-comodule, by the definition of $\mathcal{A}$. 
Next, take $\bar{m}_A:=m_A \; U^{A,A}$ and we get, by the definition of $A \; \bar{\otimes} \; A$ (\S\ref{sec:otimes-MH}), that 
\[\bar{m}_A(a\; \bar{\otimes}\; b) = m_AU^{A,A}(a\; \bar{\otimes}\; b) =m_A(a \otimes b)=ab \]
for $a, b \in A$. Since
\[
\begin{array}{rl}
\rho_A \; \bar{m}_A (a\; \bar{\otimes}\; b) ~=~ \rho_A(ab) &~=~ (ab)_{[0]} \otimes (ab)_{[1]}\\
\medskip

&\overset{\textnormal{\eqref{eq:m-A-2}}}{=} a_{[0]}b_{[0]} \otimes a_{[1]}b_{[1]}\\
\medskip

&~= (\bar{m}_A \otimes \text{Id}_H)(a_{[0]} \; \bar{\otimes} \; b_{[0]} \otimes a_{[1]}b_{[1]})\\
\medskip

& \overset{\textnormal{\eqref{eq:rhoMN}}}{=} (\bar{m}_A \otimes \text{Id}_H)\rho_{A \bar{\otimes}A}(a \; \bar{\otimes} \; b),
\end{array}
\]
we also get that $\bar{m}_A$ is a map in $\mathcal{M}^{H}$. Moreover, $\bar{m}_A$ is associative (that is, \eqref{eq:algassoc} holds) because $m_A$ is associative. 

Now, take $\bar{u}_A:= \nu_A (u_A \otimes \id_{H_s})$ and $1_{H_s}=:1$. If $x \in H_s$, then 

\begin{align*}
\rho_A \; \bar{u}_A(x) \overset{\textnormal{\eqref{eq:Hs-act}}}{=} \rho_A(1_A \triangleleft x) \hspace{-.3in}
&\hspace{.3in} \overset{\textnormal{\eqref{eq:Hs-act-Hcom}}}{=} (1_A)_{[0]} \otimes (1_A)_{[1]}x\\
\medskip
&\hspace{.25in} \overset{\textnormal{\ref{def:eps}, \eqref{eq:u-A-1}}}{=} (1_A)_{[0]} \otimes \varepsilon_t((1_A)_{[1]})x\\
\medskip
&\hspace{-.08in}\overset{\textnormal{\ref{def:eps}, \ref{prop:Hs-facts}\text{(iii)}, \eqref{eq:Delta1}, \eqref{eq:u-A-1}}}{=} (1_A)_{[0]} \otimes \varepsilon((1_A)_{[1]}1_1)1_2x\\
\medskip
&\hspace{.13in} \overset{\textnormal{\ref{prop:Hs-facts}\text{(iii)}, \eqref{eq:Delta1}}}{=} \varepsilon((1_A)_{[1]}1_1)(1_A)_{[0]} \otimes x1_2\\
\medskip
&\hspace{.29in} \overset{\textnormal{\eqref{eq:Hs-act}}}{=} (1_A \triangleleft 1_1) \otimes x 1_2\\
\medskip
& \hspace{.38in} = (\bar{u}_A \otimes \text{Id}_H)(1_1 \otimes x 1_2)\\
\medskip
&\hspace{.3in} \overset{\textnormal{\ref{prop:Hs-facts}\text{(i)}}}{=} (\bar{u}_A \otimes \text{Id}_H) \Delta(x)\\
\medskip
& \hspace{.3in} \overset{\textnormal{\S\ref{sec:unitobj-MH}}}{=} (\bar{u}_A \otimes \text{Id}_H) \rho_{H_s}(x).
\end{align*}
So, $\bar{u}_A$ is a map in $\MH$. Moreover, if $x \in H_s$ and $a \in A$, then
\[
\begin{array}{rl}
\medskip
\bar{m}_A (\bar{u}_A \; \bar{\otimes} \; \text{Id}_A)(x \; \bar{\otimes} \; a) 
&\hspace{-.1in}\overset{\textnormal{\eqref{eq:Hs-act}}}{=}~ \bar{m}_A  ((1_A \triangleleft x) \; \bar{\otimes} \; a)
\overset{\textnormal{\eqref{eq:Hs-act}}}{=} \varepsilon((1_A)_{[1]}x) \; \bar{m}_A((1_A)_{[0]} \; \bar{\otimes} \; a)\\
\medskip
&\hspace{-.09in}\overset{\textnormal{\eqref{eq:tenbar}}}{=} \varepsilon((1_A)_{[1]} x)\;  m_A\left(\varepsilon((1_A)_{[0],[1]} a_{[1]}) \; (1_A)_{[0],[0]} \otimes a_{[0]}\right)\\
\medskip
&\hspace{-.22in} \overset{\textnormal{\ref{prop:Hs-facts}\textnormal{(iii)}, \eqref{eq:u-A-1}}}{=}~ \varepsilon(x (1_A)_{[1]})\; \varepsilon((1_A)_{[0],[1]} a_{[1]}) \; (1_A)_{[0],[0]} \; a_{[0]}\\
\medskip
&\hspace{-.05in} \overset{\textnormal{\eqref{eq:Hcomod-1}}}{=}~ \varepsilon(x (1_A)_{[1],2})\; \varepsilon((1_A)_{[1],1} a_{[1]}) (1_{A})_{[0]} \;  a_{[0]}\\
\medskip
&\hspace{-.1in}  \overset{\textnormal{\ref{def:wba}\text{(iv)}}}{=}~ \varepsilon(x (1_A)_{[1]} a_{[1]}) (1_{A})_{[0]}\;  a_{[0]}\\
\medskip
&\hspace{-.05in} \overset{\textnormal{\eqref{eq:m-A-2}}}{=}~ \varepsilon(x a_{[1]})  a_{[0]}
\\
\medskip
&\hspace{-.08in} \overset{\textnormal{\S\ref{sec:unit-MH}}}{=}~ l_A(x \; \bar{\otimes} \; a).
\end{array}
\]
With a similar computation for the right unit constraint $r_A$ in $\S$\ref{sec:unit-MH}, we conclude that the unit axiom \eqref{eq:algunital} holds for the map $\bar{u}_A$.  
Therefore, $F$ sends objects to objects.

To see that $F$ sends morphisms to morphisms, take $f: A \to B \in \mathcal{A}$, which means that $f m_A = m_B(f \otimes f)$ and $f u_A = u_B$. The map $f$ must also be a morphism of $H$-comodules, so $(f \otimes f)U^{A,A} = U^{A,A}(f \tenbar f)$. Now we get:
\[
\begin{array}{rll}
\medskip
F(f)\; m_{A} \; U^{A,A} 
&=~ f \; m_{A} \; U^{A,A}
&=~ m_{B} \; (f \otimes f) \; U^{A,A}\\
&=~ m_{B}\; U^{A,A} \;  (f \; \bar{\otimes} \; f)
&=~ m_{B} \; U^{A,A} \; (F(f) \;  \bar{\otimes}  \; F(f)).
\end{array}
\]
We also get $F(f) ( \nu_A (u_{A} \otimes \id_{H_s}))  = \nu_B (u_{B} \otimes \id_{H_s})$ by the following computation. Recall by $\S$\ref{sec:Phi} that a right $H$-comodule map $f$ must also be an $H_s$-bimodule map. So, for $x \in H_s$, 
\[
\begin{array}{rl}
\medskip
f\bar{u}_A(x)&=f (\nu_A (u_{A} \otimes \id_{H_s})(x)) 
~ \overset{\textnormal{\eqref{eq:Hs-act}}}{=} ~f(1_A \triangleleft x)\\
& \hspace{-.08in} \overset{\textnormal{\S\ref{sec:Phi}}}{=}f(1_A) \triangleleft x
~=~1_B \triangleleft x ~\overset{\textnormal{\eqref{eq:Hs-act}}}{=}~ \nu_B (u_{B} \otimes \id_{H_s})(x)~=~\bar{u}_B(x).
\end{array}
\]
This ends the proof of Claim~\ref{claim:Alg.F}.
\end{proof}

\medskip

Now, take $(A, ~\bar{m}_A: A \tenbar  A \to A,~~ \bar{u}_A: H_s \to A) \in {\sf Alg}(\MH)$, 
and consider the assignment
\[
\begin{array}{rl}
\medskip
\hspace{1in} G(A, ~\bar{m}_A, ~\bar{u}_A) &:=
(A, \quad \bar{m}_A \; U_{A,A}: A \otimes A \to A, \quad \bar{u}_A\; U_0: \Bbbk \to A),\\
G(g) &:= g,
\end{array}
\]
 for $g: (A, \bar{m}_A, \bar{u}_A) \to (B, \bar{m}_B, \bar{u}_B) \in {\sf Alg}(\MH)$. Here, $U_{A,A}$ and $U_0$ are given in \eqref{eq:Ulower}. 

\begin{claim} \label{claim:Alg.G}
The assignment $G$ above yields a functor from ${\sf Alg}(\MH)$ to $\mathcal{A}$.
\end{claim} 

\begin{proof}[Proof of Claim~\ref{claim:Alg.G}] Recall from \cref{sec:U} that the forgetful functor $U:\MH \to \Vec$ is monoidal, and that $G(A,~\bar{m}_A,~\bar{u}_A):=(U(A), U(\bar{m}_A)U_{A,A}, U(\bar{u}_A)U_0).$
By \cref{prp:monfunctor}(i), we have that $G(A,~\bar{m}_A,~\bar{u}_A) \in {\sf Alg}(\Vec)$.

It remains to show that $A$ is an $H$-comodule algebra in the sense of \cref{def:A}. First, we show that $\bar{m}_A\; U_{A,A}$ satisfies \eqref{eq:m-A-2}, i.e., that the following diagram commutes, where we use $\tau$ to denote the flip map in $\Vec$:

 \[
 \xymatrix@R=1.5pc@C=4pc{
 A \otimes A \ar[rr]^{U_{A, A}} 
 \ar[d]_{\rho_{A} \otimes \rho_{A}} 
 && A \tenbar A \ar[d]_{\rho_{A \hspace{-.01in} \tenbar \hspace{-.01in} A}} \ar[r]^{\bar{m}_A} 
 & A \ar[d]^{\rho_{A}}\\
 A \otimes H \otimes A \otimes H 
 \ar[r]_{\id \otimes \tau \otimes \id} 
 & A \otimes A \otimes H \otimes H \ar[r]_{U_{A,A} \otimes m}
 & A \tenbar A \otimes H 
 \ar[r]_{\bar{m}_A \otimes \id } 
 &  A \otimes H
 }
 \]
 
 \medskip

We know that the right square commutes since $\bar{m}_A$ is an $H$-comodule map. To show that the left pentagon commutes, let $a\otimes b \in A \otimes A.$ Then
\begin{align*}
\rho_{A\tenbar A}\; U_{A,A}(a \otimes b)&
\overset{\eqref{eq:Ulower}}{=} \rho_{A\tenbar A}(\ep(a_{[1]}b_{[1]})a_{[0]}\tenbar b_{[0]})\\
& \overset{\eqref{eq:rhoMN}}{=}\ep(a_{[1]}b_{[1]})a_{[0],[0]}\tenbar b_{[0],[0]} \otimes a_{[0],[1]}b_{[0],[1]} \\
&\overset{\eqref{eq:Hcomod-1}}{=} \ep(a_{[1],2}b_{[1],2})a_{[0]}\tenbar 
b_{[0]} \otimes a_{[1],1}b_{[1],1} \\
&\overset{\textnormal{\ref{def:wba}\ref{def:wba3}}}{=} a_{[0]}\tenbar 
b_{[0]} \otimes (a_{[1]}b_{[1]})_1 \ep((a_{[1]}b_{[1]})_2)\\
&\overset{\eqref{eq:coalgcounital}}{=} a_{[0]}\tenbar 
b_{[0]} \otimes a_{[1]}b_{[1]}.
\end{align*}
Similarly,
\begin{align*}
& (U_{A,A}\otimes m)(\id \otimes \tau \otimes \id)(\rho_{A} \otimes \rho_{A})(a \otimes b)\\
&\hspace{1in}=(U_{A,A}\otimes \id)(a_{[0]}\otimes b_{[0]} \otimes a_{[1]}b_{[1]}) \\
&\hspace{1in}\overset{\eqref{eq:Ulower}}{=}\ep(a_{[0],[1]}b_{[0],[1]})a_{[0],[0]}\tenbar b_{[0],[1]} \otimes a_{[1]}b_{[1]} \\
&\hspace{1in}\overset{\eqref{eq:Hcomod-1}}{=}\ep(a_{[1],1}b_{[1],1})a_{[0]} \tenbar b_{[0]} \otimes a_{[1],2}b_{[1],2} \\
&\hspace{1in}\overset{\textnormal{\ref{def:wba}\ref{def:wba3}}}{=} a_{[0]} \tenbar b_{[0]} \otimes \ep((a_{[1]}b_{[1]})_1)(a_{[1]}b_{[1]})_2 \\
&\hspace{1in}\overset{\eqref{eq:coalgcounital}}{=} a_{[0]}\tenbar 
b_{[0]} \otimes a_{[1]}b_{[1]}.
\end{align*}

Moreover, \eqref{eq:u-A-1} holds since
\[\rho_{A}(1_{A}) 
= \rho_A \bar{u}_A U_0 (1_{\kk})
\overset{\eqref{eq:Ulower}}{=} \rho_{A} \; \bar{u}_A(1_{H_s}) 
\overset{\bar{u}_A \in \MH}{=}
(\bar{u}_A \otimes \id_H)\; \rho_{H_s} (1) 
\overset{\textnormal{\S\ref{sec:unitobj-MH}}}{=} \bar{u}_A(1_{1}) \otimes 1_{2},
\]
the latter of which belongs to $A \otimes H_t$ by  \eqref{eq:Delta1}.
Therefore, $G$ sends objects to objects.

\smallskip

Now take $g: A \to B \in {\sf Alg}(\MH)$. Since $g$ is a map in ${\sf Alg}(\MH)$, by definition $G(g)=g$ is an $H$-comodule map. Thus a straightforward calculation yields $(g \tenbar g)U_{A,A}=U_{A,A}(g\otimes g).$ We also have that $g \bar{m}_A = \bar{m}_B(g \tenbar g)$ and $g \bar{u}_A = \bar{u}_B$.  So, we get:
\[
\begin{array}{rll}
\medskip
G(g)\; \bar{m}_A \; U_{A,A} &~=~ g \; \bar{m}_A \; U_{A,A}
&=~ \bar{m}_B\; (g \tenbar g) \; U_{A,A}\\
&=~ \bar{m}_B \; U_{A,A} \; (g \otimes g)
&=~ \bar{m}_B \; U_{A,A} \; (G(g) \otimes G(g)),
\end{array}
\]
and likewise $G(g) \; \bar{u}_A \; U_0 = \bar{u}_B\; U_0$. Therefore, $G(g)$ is an algebra map in $\Vec$. Functoriality is clear by the definition of $G$. Hence, $G$ also sends morphisms to morphisms, and this ends the proof of Claim~\ref{claim:Alg.G}.
\end{proof}

\begin{claim}
 \label{claim:Alginverse}
 $F$ and $G$ are mutually inverse functors.
\end{claim}

\begin{proof}
We will show that $FG$ (resp., $GF$) is the identity functor on ${\sf Alg}(\MH)$ (resp., $\mathcal{A}$) as follows. First, note that
$FG(A,~\bar{m}_A, ~\bar{u}_A) = (A, \; \bar{m}_A\; U_{A,A} \;U^{A,A}, \; \nu_{A}(\bar{u}_A \; U_0 \otimes \text{Id}_{H_s})).$
By \eqref{eq:UMNidentity}, it is clear that $\bar{m}_A\; U_{A,A} \;U^{A,A} = \bar{m}_A$ as morphisms in $\MH$. 
Recall from Section~\ref{sec:Phi} that a morphism of right $H$-comodules is a morphism of $H_s$-bimodules. Hence, for $x \in H_s$
\[ \nu_{A} (\bar{u}_A \; U_0 \otimes \text{Id}_{H_s})(1_\kk \otimes x)  \overset{\eqref{eq:Ulower}}{=}  \nu_{A} (\bar{u}_A(1) \otimes x)  \overset{\eqref{eq:Hs-act}}{=} \bar{u}_A(1) \triangleleft x = \bar{u}_A(1 \triangleleft x) = \bar{u}_A(x).
\]
So, $\nu_{A}(\bar{u}_A \; U_0 \otimes \text{Id}_{H_s}) = \bar{u}_A$ as morphisms in ${\sf Alg}(\MH)$ as well. 

On the other hand, 
$GF(A,~m_{A}, ~u_{A}) = (A, \; m_{A}\;U^{A,A} \; U_{A,A}, \; \nu_{A}(u_{A}\otimes \text{Id}_{H_s}) U_0).$
We have that  $m_{A}\;U^{A,A} \; U_{A,A} = m_{A}$ as $\kk$-linear maps by the following computation for $a \otimes b \in A \otimes A$:
\[
\begin{array}{rll}
m_{A}\;U^{A,A} \; U_{A,A} (a \otimes b) & \hspace{-.1in}\overset{\eqref{eq:Uupper}}{=}m_AU_{A,A}(\ep(a_{[1]}b_{[1]})a_{[0]}\tenbar b_{[0]}) 
&  \overset{\eqref{eq:Ulower}}{=}m_A(\ep(a_{[1]}b_{[1]})a_{[0]}\otimes b_{[0]}) \\
&=\ep(a_{[1]}b_{[1]})a_{[0]}b_{[0]} 
&\overset{\eqref{eq:m-A-2}}{=}\ep((ab)_{[1]})(ab)_{[0]} \\
&\hspace{-.1in} \overset{\eqref{eq:Hcomod-2}}{=}m_A(a\otimes b). 
\end{array}
\]
We also get $\nu_A(u_{A}\otimes \text{Id}_{H_s}) U_0 = u_A$ due to the calculation below:
$$ \nu_{A}(u_{A}\otimes \text{Id}_{H_s}) U_0(1_\kk)
~=~ \nu_{A}(u_{A}(1_\kk) \otimes 1)
\overset{\eqref{eq:Hs-act}}{=} \varepsilon(u_A(1_\kk)_{[1]})u_A(1_\kk)_{[0]}
\overset{\eqref{eq:Hcomod-2}}{=} u_A(1_\kk).$$
Moreover, $FG(g) = g$ for any $g \in {\sf Alg}(\MH)$, and $GF(f) = f$ for any $f \in \mathcal{A}$.
\end{proof}

Therefore, the functors $F$ and $G$ (defined, respectively, in Claims~\ref{claim:Alg.F} and~~\ref{claim:Alg.G}) yield an isomorphism between the categories $\mathcal{A}:=\mathcal{A}^H$ and ${\sf Alg}(\MH)$.
\end{proof}

Now consider the following examples.

\begin{example} \label{ex:alg-Hs}
The counital subalgebra $H_s$ is  an object of $\mathcal{A}^H \cong {\sf Alg}(\MH)$, as the unit object of a monoidal category is an algebra  in the category; see Section~\ref{sec:unitobj-MH}.
\end{example}

\begin{example} \label{ex:alg-kQ}
Let $Q$ be a finite quiver. let $Q_{\ell}$ denote the paths of length $\ell$ in $Q$, write $Q_0 = \{1, \dots, n\}$, and let $e_i$ denote the trivial path at vertex $i$. Let $\kk Q$ be the path algebra of $Q$ and let $\hay$ be Hayashi's face algebra as defined in Example~\ref{ex:hay}. Recall that $\hay_t$ has $\kk$-basis $\{ \sum_{i=1}^n x_{j,i} \mid 1 \leq j \leq n\}$. In \cite[Equation 2.15]{Hayashi99b}, Hayashi notes that $\kk Q$ is a right $\hay$-comodule.
For a path $p \in Q_{\ell}$, the right coaction is given by 
\[
\rho_{\kk Q}: \kk Q \to \kk Q \otimes \mfH(Q),  \quad
p  \mapsto \textstyle\sum_{q \in Q_{\ell}} q \otimes x_{q,p}.
\]
Indeed, for a path $p \in Q_\ell$:
\[ (\id \otimes \Delta)\rho_{\kk Q}(p) = (\id \otimes \Delta)\left(\textstyle\sum_{q \in Q_\ell} q \otimes x_{q,p}\right) = \textstyle\sum_{q,r \in Q_\ell} q \otimes x_{q,r} \otimes x_{r,p}
\]
while
\[ (\rho_{\kk Q} \otimes \id)\rho_{\kk Q}(p) = (\rho_{\kk Q} \otimes \id)\left(\textstyle\sum_{q \in Q_\ell} q \otimes x_{q,p}\right) = \textstyle\sum_{r, q \in Q_\ell} r \otimes x_{r,q} \otimes x_{q, p}.
\]
So, $\rho_{\kk Q}$ is coassociative. To see that $\rho_{\kk Q}$ is counital, observe that
\[ (\id \otimes \ep)\rho_{\kk Q}(p) = (\id \otimes \ep) \left(\textstyle\sum_{q \in Q_\ell} q \otimes x_{q,p}\right) = p.
\]

In fact, $\kk Q$ is a right $\hay$-comodule algebra. First note that 
\[\rho_{\kk Q}(1_{\kk Q}) = \rho_{\kk Q} \left( \textstyle\sum_{i=1}^n e_i \right) = \textstyle\sum_{i,j=1}^n  e_j \otimes x_{j,i} = \textstyle\sum_{j=1}^n e_j \otimes \left(\sum_{i=1}^n x_{j,i}\right)
\]
and so $\rho_{\kk Q}(1_{\kk Q}) \in \kk Q \otimes \hay_t$.
Now let $p \in Q_{\ell}$ and $q \in Q_{m}$. Then if $p$ and $q$ are composable,
\begin{align*}\rho_{\kk Q}(pq) = \textstyle \sum_{r \in Q_{\ell +m}} r \otimes x_{r,pq} = \left(\textstyle\sum_{s \in Q_{\ell}} s \otimes x_{s,p}\right) \left( \textstyle\sum_{t \in Q_{m}} t \otimes x_{t,q} \right) = \rho_{\kk Q}(p) \rho_{\kk Q}(q),
\end{align*}
since each path $r \in Q_{\ell +m}$ can be written uniquely as a composition $st$ where $s \in Q_{\ell}$ and $t \in Q_m$.
If $p$ and $q$ are not composable, then
\[\rho_{\kk Q}(p) \rho_{\kk Q}(q)
= \textstyle \sum_{s \in Q_\ell, t \in Q_m} st \otimes x_{s,p} x_{t,q} = 0= \rho_{\kk Q}(pq).
\]
Therefore, $\kk Q$ and $\hay$ satisfy Definition~\ref{def:A}, so $\kk Q$ is an $\hay$-comodule algebra.
\end{example}


\section{Weak comodule coalgebras} \label{sec:coalgebras}

The goal of this section is to dualize the main result of the previous section: \cref{thm:alg}. That result pertained to comodule algebras over weak bialgebras, and here we consider comodule coalgebras over weak bialgebras. 
We then  prove the analogue of \cref{thm:alg} for these structures.

\begin{definition}[$\mathcal{C}^H =: \mc{C}$] \label{def:C}(\cite[Def.~1]{Jia}, \cite[Def.~2.1]{VZ};  \cite[Prop~3.12]{AGFLV}.) Consider the {\it category $\mathcal{C}^H$ of right $H$-comodule coalgebras} defined as follows. The objects of $\mathcal{C}^H$ are coalgebras in~${\sf Vec}_\Bbbk$, 
$(C,~\Delta_C: C \to C \otimes C , ~\ep_C: C \to \Bbbk),$
so that 
\begin{itemize}
\item the $\Bbbk$-vector space $C$ is a right $H$-comodule via $\rho_C: C \to C \otimes H, \; c \mapsto c_{[0]} \otimes c_{[1]},$ namely, the coassociativity and counit axioms hold (Eqs. \eqref{eq:Hcomod-1}, \eqref{eq:Hcomod-2});

\item the comultiplication map $\Delta_C$ is compatible with $\rho_C$ and $\rho_{C \otimes C}$ in the sense that 
\begin{eqnarray} 
\hspace{1.5in}    c_{1,[0]}\otimes c_{2,[0]} \otimes c_{1,[1]}c_{2,[1]}=c_{[0],1}\otimes c_{[0],2} \otimes c_{[1]} \hspace{0.8in}  \forall c \in C; \label{eq:delta-C}
\end{eqnarray}

\item the counit $\ep_C$ is compatible with $\rho_C$ in the sense that
\begin{eqnarray} 
\hspace{2in}    \ep_C(c_{[0]})c_{[1]}=\ep_C(c_{[0]})\eps (c_{[1]}) \hspace{1.3in}  \forall c \in C. \label{eq:ep-C}
\end{eqnarray}
\end{itemize}
The morphisms of $\CC^H$ are $\Bbbk$-linear coalgebra maps that are also $H$-comodule maps. In other words, a $\kk$-linear map $f: C \to D$ between $C, D \in \CC^H$ is a morphism in $\CC^H$ if
\begin{align}
   \label{eq:comodmap}
   f(c_{[0]}) \otimes c_{[1]}&=f(c)_{[0]} \otimes f(c)_{[1]}, \hspace{1.1in} \forall c \in C,\\
     \hspace{1.45in}  \label{eq:coalgmap-delta}f(c_1) \otimes f(c_2)& =f(c)_1 \otimes f(c)_{2} \hspace{1.3in} \forall c \in C,\\
  \label{eq:coalgmap-ep}\ep_C&=\ep_{D}f.
\end{align}
We write $\CC$ for this category when $H$ is understood.
\end{definition}

\begin{theorem} \label{thm:coalg}
Fix a weak bialgebra $H$. Then the category $\CC^H =: \mathcal{C}$ of right $H$-comodule coalgebras is isomorphic to the category {\sf Coalg}$(\MH)$ of coalgebras in $\MH$. 
\end{theorem}

\begin{proof} We achieve this result by constructing mutually inverse functors $F$ and $G$ below.

\vspace{-.1in}

\[
\begin{tikzcd}
\mathcal{C} \arrow[shift left]{rr}{F} & & {\sf Coalg}(\MH) \arrow[shift left]{ll}{G}
\end{tikzcd}
\]

First, take $(C, ~\Delta_C, ~\ep_C) \in \mathcal{C}$, and consider the assignment
\[
\begin{array}{rl}
\medskip
 F(C, ~\Delta_C, ~\ep_C) &:=
(C, \quad U_{C,C} \:\Delta_C: C \to C \tenbar C, \quad (\ep_C \otimes \id)\rho_{C}^s: C \to H_s)\\
F(f) &:= f
\end{array}
\]
 for $f: (C, \Delta_C, \ep_C) \to (D, \Delta_D, \ep_D) \in \mathcal{C}$. Here, $U_{C,C}$ and $\rho_C^s$ are given in Equations~\ref{eq:Ulower} and~\ref{eq:Hs-coact}, respectively.
 
 \begin{claim}
 \label{claim:Coalg.F}
 The assignment $F$ above yields a functor from $\CC$ to ${\sf Coalg}(\MH)$.
 \end{claim}
 
 \begin{proof}
 First we will show that $F$ takes objects of $\CC$ to objects of ${\sf Coalg}(\MH)$. Let $(C, \Delta_C, \ep_C)$ be a right $H$-comodule coalgebra. Define $$\bar{\Delta}_C:=U_{C,C}\:\Delta_C \qquad  \text{ and } \qquad \bar{\ep}_C:=(\ep_C \otimes \id)\rho_C^s.$$ We will show that $(C, \bar{\Delta}_C, \bar{\ep}_C)$ is an object of ${\sf Coalg}(\MH)$, i.e.,  that $\bar{\Delta}_C, \bar{\ep}_C$ are $H$-comodule maps, and that $\bar{\Delta}_C$ is coassociative and counital (with respect to $\bar{\ep}_C$).
 
First, we will show that $\bar{\ep}_C$ is an $H$-comodule map.
By \cref{sec:unitobj-MH}, we have $\rho_{H_s}=\Delta |_{H_s}$. For $c \in C$ we have
\begin{align*}
    \rho_{H_s}\bar{\ep}_C(c) &=\Delta |_{H_s}\bar{\ep}_C(c)& 
    &\overset{\textnormal{\ref{def:eps},\eqref{eq:Hs-coact}}}{=} \Delta |_{H_s}(\ep_C(c_{[0]})\eps(c_{[1]}))\\
    &\overset{\eqref{eq:ep-C}}{=} \Delta |_{H_s}(\ep_C(c_{[0]})c_{[1]})&
    &\overset{\eqref{eq:Hcomod-1}}{=} \ep_C(c_{[0],[0]})c_{[0],[1]} \otimes c_{[1]}\\
    &\overset{\eqref{eq:ep-C}}{=} \ep_C(c_{[0],[0]})\eps(c_{[0],[1]}) \otimes c_{[1]} \quad \text{ (applied to }c_{[0]})&
    &\overset{\textnormal{\ref{def:eps},\eqref{eq:Hs-coact}}}{=}(\bar{\ep}_C \otimes \id)(c_{[0]}\otimes c_{[1]})\\
    &=(\bar{\ep}_C \otimes \id)\rho_C(c).
\end{align*}

Next we will show that $\bar{\Delta}_C$ is an $H$-comodule map. 
Let $c \in C$. Then we have
\begin{align*}
    (\bar{\Delta}_C\otimes \id)\rho_C(c) &=(\bar{\Delta}_C\otimes \id)(c_{[0]}\otimes c_{[1]}) &
    & \\
    &=U_{C,C}{\Delta}_C(c_{[0]})\otimes c_{[1]} \\
    &=U_{C,C}(c_{[0],1}\otimes c_{[0],2})\otimes c_{[1]} \\
    &\overset{\eqref{eq:Ulower}}{=}\ep(c_{[0],1,[1]}c_{[0],2,[1]})c_{[0],1,[0]}\tenbar c_{[0],2,[0]}\otimes c_{[1]}\\
    &\overset{\eqref{eq:delta-C}}{=}\ep(c_{[0],[1]})c_{[0],[0],1}\tenbar c_{[0],[0],2}\otimes c_{[1]} \quad \text{ (applied to }c_{[0]})\\
    &\overset{\eqref{eq:Hcomod-1}}{=}\ep(c_{[1],1})c_{[0],1}\tenbar c_{[0],2}\otimes c_{[1],2} \\
    &\overset{\eqref{eq:coalgcounital}}{=}c_{[0],1}\tenbar c_{[0],2}\otimes c_{[1]}. 
\end{align*}
In the last equation, \eqref{eq:coalgcounital} is applied to $c_{[1]} \in H$, for $H \in {\sf Coalg}(\Vec)$. Similarly,
\begin{align*}
    \rho_{C\tenbar C}\bar{\Delta}_C(c)
    &=\rho_{C\tenbar C}U_{C,C}\Delta_C(c) \\
    &\overset{\eqref{eq:Ulower}}{=}\rho_{C\tenbar C}(\ep(c_{1,[1]}c_{2,[1]})c_{1,[0]}\tenbar c_{2,[0]})\\
    &\overset{\eqref{eq:rhoMN}}{=}\ep(c_{1,[1]}c_{2,[1]})c_{1,[0],[0]}\tenbar c_{2,[0],[0]} \otimes c_{1,[0],[1]}c_{2,[0],[1]}\\
    &\overset{\eqref{eq:Hcomod-1}}{=}\ep(c_{1,[1],2}c_{2,[1]})c_{1,[0]}\tenbar c_{2,[0],[0]} \otimes c_{1,[1],1}c_{2,[0],[1]} \quad \text{ (applied to }c_1)\\
    &\overset{\eqref{eq:Hcomod-1}}{=}\ep(c_{1,[1],2}c_{2,[1],2})c_{1,[0]}\tenbar c_{2,[0]} \otimes c_{1,[1],1}c_{2,[1],1} \quad \text{ (applied to }c_2)\\
    &\overset{\textnormal{\ref{def:wba}\ref{def:wba3}}}{=}\ep((c_{1,[1]}c_{2,[1]})_2)c_{1,[0]}\tenbar c_{2,[0]} \otimes (c_{1,[1]}c_{2,[1]})_1 \\
    &\overset{\eqref{eq:coalgcounital}}{=}c_{1,[0]}\tenbar c_{2,[0]} \otimes c_{1,[1]}c_{2,[1]} \\
    &\overset{\eqref{eq:delta-C}}{=} c_{[0],1}\tenbar c_{[0],2} \otimes c_{[1]},
\end{align*}
so the diagram commutes.
 
Now we will show that $\bar{\Delta}_C$ is counital, i.e., that the following diagram commutes, where $l_C,r_C$ are the left and right unit isomorphisms defined in \cref{sec:unit-MH}:
 \[
 \begin{tikzcd}
 H_s \tenbar C  &  C \arrow[swap]{l}{l_C^{-1}} \arrow{r}{r_C^{-1}}\arrow{d}{\bar{\Delta}_C} & C \tenbar H_s  \\
 {} & C \tenbar C \arrow{lu}{\bar{\ep}_C \tenbar \id}\arrow[swap]{ru}{\id \tenbar \bar{\ep}_C} & {}
 \end{tikzcd}
 \]
We will give a proof that the right hand side of the diagram commutes; the left is similar. Let $c\in C$. We will show that $r_C(\id \tenbar \bar{\ep}_C)\bar{\Delta}_C(c) = c.$
 \begin{align*}
r_C(\id \tenbar \bar{\ep}_C)\bar{\Delta}_C(c) &=r_C(\id \tenbar \bar{\ep}_C)U_{C,C}\Delta_C(c)\\  
&\overset{\eqref{eq:Ulower}}{=}r_C(\id \tenbar \bar{\ep}_C)(\ep(c_{1,[1]}c_{2,[1]})c_{1,[0]} \tenbar c_{2,[0]}) \\
&\overset{\eqref{eq:Hs-coact}, \textnormal{\ref{def:eps}}}{=}r_C\left(\ep(c_{1,[1]}c_{2,[1]})c_{1,[0]} \tenbar \eps(c_{2,[0],[1]})\ep_C(c_{2,[0],[0]})\right) \\
&\overset{\eqref{eq:ep-C}}{=} r_C\left(\ep(c_{1,[1]}c_{2,[1]})c_{1,[0]} \tenbar c_{2,[0],[1]}\ep_C(c_{2,[0],[0]})\right) \quad \text{ (applied to }c_{2,[0]}) \\
&\overset{\eqref{eq:Hcomod-1}}{=}r_C\left(\ep_C(c_{2,[0]})\ep(c_{1,[1]}c_{2,[1],2})c_{1,[0]} \tenbar c_{2,[1],1}) \quad \text{ (applied to }c_2\right)\\
&\overset{*}{=}r_C\left(\ep_C(c_{2,[0]})\ep(c_{1,[1]}\eps(c_{2,[1]})_2)c_{1,[0]} \tenbar \eps(c_{2,[1]})_1\right) \\
&\overset{\textnormal{\ref{prop:Hs-facts}\textnormal{(i)}}}{=} r_C\left(\ep_C(c_{2,[0]})\ep(c_{1,[1]}\eps(c_{2,[1]})1_2)c_{1,[0]} \tenbar 1_1\right) \quad  \\
&\overset{\textnormal{\ref{def:eps}}}{=}r_C\left(\ep_C(c_{2,[0]})c_{1,[0]} \tenbar  \eps(c_{1,[1]}\eps(c_{2,[1]}))\right)\\
&\overset{\eqref{eq:ep-C}}{=}r_C\left(\ep_C(c_{2,[0]})c_{1,[0]} \tenbar  \eps(c_{1,[1]}c_{2,[1]})\right) \quad \text{ (applied to }c_2)\\
&\overset{\eqref{eq:delta-C}}{=}r_C\left(\ep_C(c_{[0],2})c_{[0],1} \tenbar  \eps(c_{[1]})\right)\\
&\overset{\eqref{eq:coalgcounital}}{=}r_C(c_{[0]}\tenbar \eps(c_{[1]}))\\
&\overset{\textnormal{\S\ref{sec:unit-MH}}}{=}\ep(c_{[0],[1]}\eps(c_{[1]}))c_{[0],[0]} \\
&\overset{\eqref{eq:Hcomod-1}}{=}\ep(c_{[1],1}\eps(c_{[1],2}))c_{[0]} \\
&\overset{\eqref{eq:g}}{=}\ep(c_{[1]})c_{[0]} \\
&\overset{\eqref{eq:Hcomod-2}}{=}c. 
\end{align*} 
 To justify $*$:
 \begin{align*}
     c_1 \otimes \ep_C(c_{2,[0]})c_{2,[1],1} \otimes c_{2,[1],2} 
     &=c_1 \otimes\Delta(\ep_C(c_{2,[0]})c_{2,[1]}) \\
     &\overset{\eqref{eq:ep-C}}{=}c_1 \otimes\Delta(\ep_C(c_{2,[0]})\eps(c_{2,[1]})) \\
     &=c_1 \otimes \ep_C(c_{2,[0]})\eps(c_{2,[1]})_1 \otimes \eps(c_{2,[1]})_2.
 \end{align*}
 
Finally, we will show that $\bar{\Delta}_C$ is coassociative. 
Take $c \in C$, and compute:
 \begin{align*}
     & (\id \tenbar \bar{\Delta}_C)\bar{\Delta}_C (c)\\
     &=(\id \tenbar \bar{\Delta}_C)U_{C,C}\Delta_C(c) \\
     &\overset{\eqref{eq:Ulower}}{=} (\id \tenbar \bar{\Delta}_C)(\ep(c_{1,[1]}c_{2,[1]})c_{1,[0]}\tenbar c_{2,[0]}) \\
     &=(\id \tenbar U_{C,C}\Delta_C)(\ep(c_{1,[1]}c_{2,[1]})c_{1,[0]}\tenbar c_{2,[0]}) \\
     &=\ep(c_{1,[1]}c_{2,[1]})c_{1,[0]}\tenbar U_{C,C}(c_{2,[0],1} \otimes c_{2,[0],2}) \\
     &\overset{\eqref{eq:Ulower}}{=}\ep(c_{1,[1]}c_{2,[1]})c_{1,[0]}\tenbar \ep(c_{2,[0],1,[1]}c_{2,[0],2,[1]})c_{2,[0],1,[0]} \tenbar c_{2,[0],2,[0]} \\
     &=\ep(c_{1,[1]}c_{2,[1]})\ep(c_{2,[0],1,[1]}c_{2,[0],2,[1]})c_{1,[0]}\tenbar c_{2,[0],1,[0]} \tenbar c_{2,[0],2,[0]} \\
     &\overset{\eqref{eq:delta-C}}{=}\ep(c_{1,[1]}c_{2,1,[1]}c_{2,2,[1]})\ep(c_{2,1,[0],[1]}c_{2,2,[0],[1]})c_{1,[0]}\tenbar c_{2,1,[0],[0]} \tenbar c_{2,2,[0],[0]} \text{ (applied to }c_2)\\
     &\overset{\eqref{eq:coalgcoassoc}}{=}\ep(c_{1,[1]}c_{2,[1]}c_{3,[1]})\ep(c_{2,[0],[1]}c_{3,[0],[1]})c_{1,[0]}\tenbar c_{2,[0],[0]} \tenbar c_{3,[0],[0]} \\
     &\overset{\eqref{eq:Hcomod-1}}{=}\ep(c_{1,[1]}c_{2,[1],2}c_{3,[1]})\ep(c_{2,[1],1}c_{3,[0],[1]})c_{1,[0]}\tenbar c_{2,[0]} \tenbar c_{3,[0],[0]}\quad  \text{ (applied to }c_2)\\
     &\overset{\eqref{eq:Hcomod-1}}{=}\ep(c_{1,[1]}c_{2,[1],2}c_{3,[1],2})\ep(c_{2,[1],1}c_{3,[1],1})c_{1,[0]}\tenbar c_{2,[0]} \tenbar c_{3,[0]}\quad \text{ (applied to }c_3)\\
     &\overset{\textnormal{\ref{def:wba}\ref{def:wba3}}}{=}\ep(c_{1,[1]}(c_{2,[1]}c_{3,[1]})_2)\ep((c_{2,[1]}c_{3,[1]})_1)c_{1,[0]}\tenbar c_{2,[0]} \tenbar c_{3,[0]} \\
     &\overset{\textnormal{\ref{def:wba}\ref{def:wba4}}}{=}\ep(c_{1,[1]}c_{2,[1]}c_{3,[1]})c_{1,[0]}\tenbar c_{2,[0]} \tenbar c_{3,[0]}.
 \end{align*}
 \noindent On the other hand, we have
  \begin{align*}
     &(\bar{\Delta}_C \tenbar \id)\bar{\Delta}_C (c)\\
     &=(\bar{\Delta}_C \tenbar \id)(\ep(c_{1,[1]}c_{2,[1]})c_{1,[0]}\tenbar c_{2,[0]}) \\
     &=(U_{C,C}\Delta_C \tenbar \id)(\ep(c_{1,[1]}c_{2,[1]})c_{1,[0]}\tenbar c_{2,[0]}) \\
     &=\ep(c_{1,[1]}c_{2,[1]})U_{C,C}(c_{1,[0],1}\otimes c_{1,[0],2}) \tenbar c_{2,[0]} \\
     &\overset{\eqref{eq:Ulower}}{=}\ep(c_{1,[1]}c_{2,[1]})\ep(c_{1,[0],1,[1]}c_{1,[0],2,[1]})c_{1,[0],1,[0]} \tenbar c_{1,[0],2,[0]} \tenbar c_{2,[0]} \\
     &\overset{\eqref{eq:delta-C}}{=}\ep(c_{1,1,[1]}c_{1,2,[1]}c_{2,[1]})\ep(c_{1,1,[0],[1]}c_{1,2,[0],[1]})c_{1,1,[0],[0]} \tenbar c_{1,2,[0],[0]} \tenbar c_{2,[0]}\;\; \text{ (applied to }c_1)\\
     &\overset{\eqref{eq:coalgcoassoc}}{=}\ep(c_{1,[1]}c_{2,[1]}c_{3,[1]})\ep(c_{1,[0],[1]}c_{2,[0],[1]})c_{1,[0],[0]} \tenbar c_{2,[0],[0]} \tenbar c_{3,[0]} \\
     &\overset{\eqref{eq:Hcomod-1}}{=}\ep(c_{1,[1],2}c_{2,[1]}c_{3,[1]})\ep(c_{1,[1],1}c_{2,[0],[1]})c_{1,[0]} \tenbar c_{2,[0],[0]} \tenbar c_{3,[0]} \quad \text{ (applied to }c_1)\\
     &\overset{\eqref{eq:Hcomod-1}}{=}\ep(c_{1,[1],2}c_{2,[1],2}c_{3,[1]})\ep(c_{1,[1],1}c_{2,[1],1})c_{1,[0]} \tenbar c_{2,[0]} \tenbar c_{3,[0]}\quad \text{ (applied to }c_2)\\
     &\overset{\textnormal{\ref{def:wba}\ref{def:wba3}}}{=}\ep((c_{1,[1]}c_{2,[1]})_2c_{3,[1]})\ep((c_{1,[1]}c_{2,[1]})_1)c_{1,[0]} \tenbar c_{2,[0]} \tenbar c_{3,[0]} \\
     &\overset{\textnormal{\ref{def:wba}\ref{def:wba4}}}{=}\ep(c_{1,[1]}c_{2,[1]}c_{3,[1]})c_{1,[0]} \tenbar c_{2,[0]} \tenbar c_{3,[0]}.
 \end{align*}
 Therefore, $\bar{\Delta}_C$ is coassociative.
 
 Now suppose that $f: C \to D$ is a morphism in $\CC$. We will show that $F(f):=f$ is also a morphism in ${\sf Coalg}(\MH).$ 
 Recall from \cref{sec:otimes-MH} that $(f\tenbar f)(c \tenbar c'):=f(c)\tenbar f(c')$ for $c \tenbar c' \in C\tenbar C$. We start with the first diagram. For any $c \in C$, we have
{\small \begin{align*}
     (f\tenbar f)\bar{\Delta}_C(c) &=(f\tenbar f)U_{C,C}\Delta_C(c) &
     &\overset{\eqref{eq:Ulower}}{=}(f\tenbar f)(\ep(c_{1,[1]}c_{2,[1]})c_{1,[0]}\tenbar c_{2,[0]}) \\
     &=\ep(c_{1,[1]}c_{2,[1]})f(c_{1,[0]})\tenbar f(c_{2,[0]}) &
     &\overset{\eqref{eq:comodmap}}{=}\ep(f(c_{1})_{[1]}f(c_{2})_{[1]})f(c_{1})_{[0]}\tenbar f(c_{2})_{[0]} \\
     &\overset{\eqref{eq:coalgmap-delta}}{=}\ep(f(c)_{1,[1]}f(c)_{2,[1]})f(c)_{1,[0]}\tenbar f(c)_{2,[0]} &
     &= U_{D,D}(f(c)_1 \otimes f(c)_2)\\
     &=U_{D,D}\Delta_D f(c) &
     &=\bar{\Delta}_Df(c).
 \end{align*} }
 We also have
 \begin{align*}
     \bar{\ep}_{D}f(c) 
     &\overset{\eqref{eq:Hs-coact},\textnormal{\ref{def:eps}}}{=} \eps(f(c)_{[1]})\ep_{D}(f(c)_{[0]}) 
     & \hspace{-.1in} \overset{\eqref{eq:comodmap}}{=}\eps(c_{[1]})\ep_{D}(f(c_{[0]})) 
     &\overset{\eqref{eq:coalgmap-ep}}{=}\eps(c_{[1]})\ep_{C}(c_{[0]}) 
     &=\bar{\ep}_{C}(c).
 \end{align*}
We have shown that $F(f):=f$ is a morphism in ${\sf Coalg}(\MH)$. It is clear that $F$ is functorial since $F(fg)=fg=F(f)F(g)$ and $F(\id_C)=\id_C=\id_{F(C)}$ for any morphisms $f,g$ in $\CC$ and any $C \in \CC$. 
 \end{proof}
 
Now take $(C, ~\bar{\Delta}_C, ~\bar{\ep}_C) \in {\sf Coalg}(\MH)$, and consider the assignment
\[
\begin{array}{rl}
\medskip
 G(C, ~\bar{\Delta}_C, ~\bar{\ep}_C) &:=
(C, \quad U^{C,C} \:\bar{\Delta}_C: C \to C \otimes C, \quad U^0\: \bar{\ep}_C: C \to \kk)\\
G(\bar{f}) &:= \bar{f}
\end{array}
\]
 for $\bar{f}: (C, \bar{\Delta}_C, \bar{\ep}_C) \to (D, \bar{\Delta}_D, \bar{\ep}_D) \in {\sf Coalg}(\MH)$. Here, $U^{C,C}$ and $U^0$ are given in \eqref{eq:Ulower}.
 
 \begin{claim}
 \label{claim:Coalg.G}
 The assignment $G$ given above yields a functor from ${\sf Coalg}(\MH)$ to $\CC$.
 \end{claim}
 
 \begin{proof}
To begin, define $$\Delta_C:=U^{C,C}\bar{\Delta}_C \quad \quad \text{ and } \quad \quad \ep_C:=U^0 \bar{\ep}_C.$$ Then by \cref{prp:monfunctor}(ii) and the fact that the forgetful functor $U$ is comonoidal (see \cref{sec:U}), we have that $(C, \Delta_C, \ep_C)$ is a coalgebra  in ${\sf Vec}_{\kk}$. It remains to check that $C$ is a right $H$-comodule coalgebra; that is, $C$ satisfies \eqref{eq:delta-C} and \eqref{eq:ep-C}.
 
 It is easily checked that \eqref{eq:delta-C} is equivalent to the commutativity of the outer square of the following diagram, where $\tau$  is the flip map in $\Vec$ and $m$ is the multiplication in $H$:

 \[
\begin{tikzcd}
 C \arrow[]{rrrrr}{\rho_C} \arrow[swap]{d}{\bar{\Delta}_C} 
 &&&&& C \otimes H \arrow{d}{\bar{\Delta}_C \otimes \id} \\
 C \tenbar C 
 \arrow[]{rrrrr}{\rho_{C \bar{\otimes} C}} \arrow[swap]{d}{U^{C,C}}
 &&&&& C \tenbar C \otimes H \arrow[]{d}{U^{C,C} \otimes \id}\\
  C \otimes C \arrow[]{r}{\rho_C \otimes \rho_C} 
 &C \otimes H \otimes C \otimes H \arrow{rr}{\id \otimes \tau \otimes \id} && C \otimes C \otimes H \otimes H \arrow[]{rr}{\id \otimes \id \otimes m}
 && C \otimes C \otimes H.
 \end{tikzcd}
 \]
 
 \bigskip

\noindent Since $\bar{\Delta}_C \in {\sf Coalg}(\MH)$, it is an $H$-comodule morphism, which leads to commutativity of the top square.
 To show commutativity of the rest of the diagram, let  $c \tenbar d \in C \tenbar C$, and consider the computation below:
 \begin{align*}
 (U^{C,C}\otimes \id)\rho_{C\bar{\otimes}C}(c \tenbar d) &=(U^{C,C}\otimes \id)(c_{[0]}\tenbar d_{[0]} \otimes c_{[1]}d_{[1]}) \\
&\overset{\eqref{eq:Uupper}}{=} c_{[0]}\otimes d_{[0]} \otimes c_{[1]}d_{[1]} \\
 &=(\id \otimes \id \otimes m)(\id \otimes \tau \otimes \id)(\rho_C \otimes \rho_C)(c\otimes d) \\
 &=(\id \otimes \id \otimes m)(\id \otimes \tau \otimes \id)(\rho_C \otimes \rho_C)U^{C,C}(c\tenbar d).
 \end{align*}
 Therefore, \eqref{eq:delta-C} holds. 
 
 Now we will show that \eqref{eq:ep-C} holds. Recall from \eqref{eq:Hs-coact} and \cref{def:eps} that for $c \in C,$ we have $\rho_C^s(c)=c_{[0]} \otimes \eps(c_{[1]}).$ 
Writing $\ep_C$ in terms of $\bar{\ep}_C$ yields that \eqref{eq:ep-C} holds if and only~if
 \[(U^0 \otimes \id)(\bar{\ep}_C \otimes \id) \rho(c)=(U^0 \otimes \id)(\bar{\ep}_C \otimes \id)\rho_C^s(c), \quad \forall c\in C,\]
 where $U^0$ is defined in \eqref{eq:Uupper}. In fact, we will show that if $(C, \bar{\Delta}_C, \bar{\ep}_C) \in {\sf Coalg}(\MH)$, then
 \begin{equation}
 \label{eq:barepC}
 (U^0 \otimes \id)(\bar{\ep}_C \otimes \id) \rho(c)=\bar{\ep}_C(c)=(U^0 \otimes \id)(\bar{\ep}_C \otimes \id)\rho_C^s(c), \quad \forall c\in C.
 \end{equation}
 
\noindent  Recall that since $\bar{\ep}_C \in {\sf Coalg}(\MH),$ it is an $H$-comodule map. In other words,
 \begin{equation}
 \label{eq:barep-Hcomodmap}
     \bar{\ep}_C(c_{[0]}) \otimes c_{[1]} = \rho_{H_s}\bar{\ep}_C(c) = \bar{\ep}_C(c)_1 \otimes \bar{\ep}_C(c)_2, \quad \forall c\in C,
 \end{equation}
  where $\rho_{H_s}$ is defined in \cref{sec:unitobj-MH}. Therefore,
 \[
 \begin{array}{rll}
 (U^0 \otimes \id)(\bar{\ep}_C \otimes \id) \rho(c)&=(U^0 \otimes \id)(\bar{\ep}_C(c_{[0]})\otimes c_{[1]})
 &\overset{\eqref{eq:Uupper}}{=}\ep(\bar{\ep}_C(c_{[0]})) c_{[1]} \\
 &\overset{\eqref{eq:barep-Hcomodmap}}{=}\ep(\bar{\ep}_C(c)_1) \bar{\ep}_C(c)_2 
 &\overset{\eqref{eq:coalgcounital}}{=} \bar{\ep}_C(c) \\
 &\overset{\eqref{eq:a}}{=} \eps(\bar{\ep}_C(c)) 
  &\overset{\eqref{eq:coalgcounital}}{=} \eps(\ep(\bar{\ep}_C(c)_1) \bar{\ep}_C(c)_2) \\
  \smallskip
  &=\ep(\bar{\ep}_C(c)_1)\eps(\bar{\ep}_C(c)_2) 
  &\overset{\eqref{eq:barep-Hcomodmap}}{=}\ep(\bar{\ep}_C(c_{[0]}))\eps(c_{[1]}) \\
  \smallskip
  &= (U^0 \otimes \id)(\bar{\ep}_C(c_{[0]})\otimes \eps(c_{[1]}))
  \smallskip
   &= (U^0 \otimes \id)(\bar{\ep}_C \otimes \id)\rho_C^s(c).
 \end{array} 
 \]
  
  \noindent So, \eqref{eq:barepC} holds, and we have shown that $G$ takes objects of ${\sf Coalg}(\MH)$ to objects of $\CC$. 
  
  It remains to show that $G$ takes morphisms of ${\sf Coalg}(\MH)$ to morphisms of $\CC$, since if so, it is clear that $G$ is functorial. Let $\bar{f}:C \to D \in {\sf Coalg}(\MH)$. Then $G(\bar{f})=\bar{f}$ is an $H$-comodule map. It remains to check that it is a coalgebra map: conditions \eqref{eq:coalgmap-delta} and \eqref{eq:coalgmap-ep}.
 Notice that $U^{C,C}(\bar{f} \tenbar \bar{f})=(\bar{f} \otimes \bar{f})U^{C,C}$ by definition of $\tenbar$ and the fact that $U^{C,C}$ is inclusion; see \eqref{eq:f-tenbar-g} and \eqref{eq:Uupper}. Then  \eqref{eq:coalgmap-delta} holds as 
$$
(\bar{f} \otimes \bar{f})\Delta_C = (\bar{f} \otimes \bar{f})U^{C,C}\bar{\Delta}_C 
\overset{}{=} U^{C,C}(\bar{f} \tenbar \bar{f})\bar{\Delta}_C 
\overset{}{=} U^{C,C}\bar{\Delta}_D \bar{f} \quad 
=\Delta_D \bar{f}.
$$
 On the other hand, \eqref{eq:coalgmap-ep}, we have
$\ep_{D} \bar{f}  = U^0\bar{\ep}_{D} \bar{f} = U^0 \bar{\ep}_{C} = \ep_{C},$
 where the second equality holds because $\bar{f} \in {\sf Coalg}(\MH)$. Therefore, $G$ is well-defined on morphisms.
 \end{proof}
 
 \begin{claim}
 \label{claim:Coalg.inverse}
 The functors $F$ and $G$ are mutually inverse.
 \end{claim}
 
 \begin{proof}
 Clearly, $FG$ and $GF$ are the identity on morphisms. It remains to show that they act as the identity on objects. Suppose that $(C, \bar{\Delta}_C,\bar{\ep}_C) \in {\sf Coalg}(\MH)$. Then
 \begin{align*}
FG(C, \bar{\Delta}_C,\bar{\ep}_C) &= F(C,\; U^{C,C}\bar{\Delta}_C, \; U^0\bar{\ep}_C) &
&= (C,\;  U_{C,C}U^{C,C}\bar{\Delta}_C,\; (U^0\bar{\ep}_C \otimes \id) \rho^s_C) \\
&\overset{\eqref{eq:UMNidentity}}{=} (C, \; \bar{\Delta}_C, \; (U^0\bar{\ep}_C \otimes \id) \rho^s_C) &
&\overset{\eqref{eq:barepC}}{=} (C, \bar{\Delta}_C,\bar{\ep}_C).
 \end{align*}
 
 Now suppose that $(C, \Delta_C, \ep_C) \in \CC$. We have
 \begin{align*}
 GF(C, \Delta_C, \ep_C) &=G(C, \; U_{C,C}\Delta_C, \; (\ep_C \otimes \id)\rho^s_C) \\
 &=(C, \; U^{C,C}U_{C,C}\Delta_C,\;  U^0(\ep_C \otimes \id)\rho^s_C) \; = \; (C, \Delta_C, \ep_C),
\end{align*}
where the last equality is justified by the following calculations. For any $c \in C$, we have
\[
\begin{array}{rll}
U^{C,C}U_{C,C}\Delta_C(c) &=U^{C,C}U_{C,C}(c_1 \otimes c_2) 
&\overset{\eqref{eq:Ulower}}{=}U^{C,C}(\ep(c_{1,[1]}c_{2,[1]})c_{1,[0]}\tenbar c_{2,[0]}) \\
&\overset{\eqref{eq:Uupper}}{=}\ep(c_{1,[1]}c_{2,[1]})c_{1,[0]}\otimes c_{2,[0]} 
&\overset{\eqref{eq:delta-C}}{=}\ep(c_{[1]})c_{[0],1}\otimes c_{[0],2} \\
&=\Delta_C(\ep(c_{[1]})c_{[0]}) 
&\overset{\eqref{eq:Hcomod-2}}{=}\Delta_C(c). 
\end{array}
\]
Similarly,
\[
\begin{array}{rll}
U^0(\ep_C \otimes \id)\rho^s_C(c) &= U^0(\ep_C(c_{[0]})\eps(c_{[1]})) 
&\overset{\eqref{eq:ep-C}}{=}U^0(\ep_C(c_{[0]})c_{[1]}) \\
&\overset{\eqref{eq:Uupper}}{=}\ep(\ep_C(c_{[0]})c_{[1]}) 
&=\ep_C(c_{[0]}\ep(c_{[1]})) \\
&\overset{\eqref{eq:Hcomod-2}}{=}\ep_C(c).
\end{array}
\]

\vspace{-.2in}

\qedhere
 \end{proof}
 
Therefore, the functors $F$ and $G$ (defined, respectively, in Claims~\ref{claim:Coalg.F} and~~\ref{claim:Coalg.G}) yield an isomorphism between the categories $\mathcal{C}:=\mathcal{C}^H$ and ${\sf Coalg}(\MH)$.
\end{proof}

Now consider the following examples.

\begin{example} \label{ex:coalg-Hs}
The counital subalgebra $H_s$ is  an object of $\mathcal{C}^H \cong {\sf Coalg}(\MH)$, as the unit object of a monoidal category is a coalgebra  in the category; see Section~\ref{sec:unitobj-MH}.
\end{example}

\begin{example} \label{ex:coalg-kQ}
Let $Q$ be a finite quiver. Let $Q_{\ell}$ denote the paths of length $\ell$ in $Q$, let $Q_0 = \{1, 2, \dots ,n \}$, and write $e_i$ for the trivial path at vertex $i$.
Let $\hay$ denote Hayashi's face algebra as defined in Example~\ref{ex:hay}. Recall from Example~\ref{ex:alg-kQ} that the path algebra $\kk Q$ is an $\hay$-comodule algebra via the right coaction $\rho_{\kk Q}(p) = \sum_{q \in Q_{\ell}} q \otimes x_{q,p}$ for $p \in Q_{\ell}$.
We now show that this coaction makes the path coalgebra $\kk Q$ an $\hay$-comodule coalgebra.

The coalgebra structure that we consider on $\kk Q$ is given as follows.
If $i \in  Q_0$ then $\Delta_{\kk Q}(e_i) = e_i \otimes e_i$ and $\ep_{\kk Q}(e_i) = 1$.
For $\ell \geq 1$ and for a path $p \in Q_{\ell}$, write $p = p_1 p_2 \cdots p_{\ell}$ where $p_1, \dots, p_{\ell} \in Q_1$. Then 
\[ \Delta_{\kk Q}(p) = (e_{s(p)} \otimes p) + \left(\textstyle \sum_{i=1}^{\ell-1} p_1 \cdots p_{i} \otimes p_{i+1}\cdots p_{\ell}\right) + (p \otimes e_{t(p)})
\]
and $\ep_{\kk Q}(p) = 0$.

First, we verify the compatibility of the comultiplication $\Delta_{\kk Q}$ with the coactions  $\rho_{\kk Q}$ and $\rho_{\kk Q \otimes \kk Q}$. If $i \in Q_0$, then
\begin{align*} 
 \rho_{\kk Q \otimes \kk Q} (\Delta_{\kk Q}(e_i)) &= \rho_{\kk Q \otimes \kk Q}(e_i \otimes e_i)  = \textstyle\sum_{j,k \in Q_0}   e_j \otimes e_k \otimes x_{j,i} x_{k,i} = \textstyle\sum_{j \in Q_0} e_j \otimes e_j \otimes x_{j,i} \\
&= \textstyle (\Delta_{\kk Q} \otimes \id) \left( \sum_{j \in Q_0} e_j \otimes x_{j,i} \right) = (\Delta_{\kk Q} \otimes \id)\rho_{\kk Q}(e_i).
\end{align*}
Now $p \in Q_{\ell}$ for some $\ell \geq 1$. Then we have
\begin{align*} \rho_{\kk Q \otimes \kk Q} (\Delta_{\kk Q}(p)) &= \textstyle \rho_{\kk Q \otimes \kk Q}\left( e_{s(p)} \otimes p + \left(\textstyle\sum_{i=1}^{\ell-1} p_1 \cdots p_{i} \otimes p_{i+1}\cdots p_{\ell}\right) + p \otimes e_{t(p)} \right) \\
&= \textstyle \sum_{\substack{j \in Q_0 \\ q \in Q_{\ell}}}   e_j \otimes q \otimes x_{j, s(p)} x_{q,p} + \textstyle\sum_{i=1}^{\ell-1} \sum_{\substack{r \in Q_i \\ u \in Q_{\ell-i}}} r \otimes u \otimes x_{r, p_1 \cdots p_i} x_{u, p_{i+1} \cdots p_{\ell}} \\
& \hspace{.2in} +\textstyle \sum_{\substack{v \in Q_\ell \\ k \in Q_0}} v \otimes e_{k} \otimes x_{v,p} x_{k, t(p)} \\
&= \sum_{\substack{q \in Q_{\ell}}}   e_{s(q)} \otimes q \otimes  x_{q,p} + \sum_{i=1}^{\ell-1} \sum_{\substack{r \in Q_i, u \in Q_{\ell-i} \\ \text{with } t(r) = s(u)}} \hspace{-.2in} r \otimes u \otimes x_{ru, p} + \sum_{v \in Q_\ell} v \otimes e_{t(v)} \otimes x_{v,p}.
\end{align*}
Further,
\begin{align*}
(\Delta_{\kk Q} \otimes \id)\rho_{\kk Q}(p) &=  \textstyle (\Delta_{\kk Q} \otimes \id) \left( \sum_{q \in Q_{\ell}} q \otimes x_{q,p} \right) \\
& \hspace{-.4in} = \textstyle \sum_{q \in Q_{\ell}} \hspace{-.05in} \left( e_{s(q)} \otimes q \otimes x_{q,p} + \left(\sum_{i=1}^{\ell-1} q_1 \cdots q_{i} \otimes q_{i+1}\cdots q_{\ell} \otimes x_{q,p}\right) + q \otimes e_{t(q)} \otimes x_{q,p} \right).
\end{align*}
So, for each $p \in Q_{\ell}$ we have
$\rho_{\kk Q \otimes \kk Q} (\Delta_{\kk Q}(p)) = (\Delta_{\kk Q} \otimes \id)\rho_{\kk Q}(p),$
as desired.

For compatibility of the counit $\ep_{\kk Q}$ with the coaction, first let $i \in Q_0$. Then
\[ \textstyle (\ep_{\kk Q} \otimes \id)(\rho_{\kk Q}(e_i)) = (\ep_{\kk Q} \otimes \id)\left( \sum_{j \in Q_0} e_j \otimes x_{j,i}\right) = \sum_{j \in Q_0} x_{j,i},
\]
while
\[ \textstyle (\ep_{\kk Q} \otimes \eps)(\rho_{\kk Q}(e_i)) = (\ep_{\kk Q} \otimes \eps)\left( \sum_{j \in Q_0} e_j \otimes x_{j,i} \right) = \sum_{j \in Q_0} \eps(x_{j,i}) = \eps(x_{i,i}) = \sum_{j \in Q_0} x_{j,i}.
\]
Now let $\ell \geq 1$ and let $p \in Q_{\ell}$. Then
\[  \textstyle  (\ep_{\kk Q} \otimes \id)(\rho_{\kk Q}(p)) = (\ep_{\kk Q} \otimes \id)\left( \sum_{q \in Q_{\ell}} q \otimes x_{q,p}\right) = 0 = (\ep_{\kk Q} \otimes \eps)(\rho_{\kk Q}(p)).
\]
Hence, by Definition~\ref{def:C}, the path coalgebra $\kk Q$ is an $\hay$-comodule coalgebra.
\end{example}


\section{Weak comodule Frobenius algebras} \label{sec:Frobenius}

We now define comodule Frobenius algebras over weak bialgebras and prove a result analogous to Theorems~\ref{thm:alg} and~\ref{thm:coalg}. Recall the category $\mathcal{A}^H$ of right $H$-comodule algebras and the category $\mathcal{C}^H$ of right $H$-comodule coalgebras from Definitions~\ref{def:A} and~\ref{def:C}, respectively. 
\begin{definition}[$\mathcal{F}^H =: \mc{F}$] \label{def:F}  Consider the {\it category $\mathcal{F}^H$ of right $H$-comodule Frobenius algebras} defined as follows. The objects of $\mathcal{F}^H$ are Frobenius algebras in~${\sf Vec}_\Bbbk$, 
$$(A, \; \; m_A: A \otimes A \to A, \; \; u_A: \kk \to A, \; \; \Delta_A: A \to A \otimes A , \; \; \ep_A: A \to \Bbbk),$$
so that $(A,m_A,u_A) \in \mathcal{A}^H$ and $(A,\Delta_A,\ep_A) \in \mathcal{C}^H$.
The morphisms of $\mathcal{F}^H$ are $\kk$-linear maps that simultaneously belong to $\mathcal{A}^H$ and $\mathcal{C}^H$.
When the weak bialgebra $H$ is clear from context, we denote this category by $\mc{F}$.
\end{definition}

\begin{theorem} \label{thm:FrobAlg}
The category $\mathcal{F}$ of right $H$-comodule Frobenius algebras is isomorphic to the category {\sf FrobAlg}$(\MH)$ of Frobenius algebras in $\MH$. 
\end{theorem}

\begin{proof} Similar to the proofs of Theorems~\ref{thm:alg} and~\ref{thm:coalg}, we achieve this result by constructing mutually inverse functors $F$ and $G$ below.

\vspace{-.1in}

\[
\begin{tikzcd}
\mathcal{F} \arrow[shift left]{rr}{F} & & {\sf FrobAlg}(\MH) \arrow[shift left]{ll}{G}
\end{tikzcd}
\]

First, take $(A,m_A, u_A, \Delta_A, \ep_A) \in \mathcal{F}$, and consider the assignment
\[
\begin{array}{rl}
\medskip
 F(A,m_A, u_A, \Delta_A, \ep_A) &:=
(A,\; m_A U^{A,A},\; \nu_A  (u_A \otimes \id_{H_s}),\; U_{A,A} \:\Delta_A, \;(\ep_A \otimes \id)\rho_{A}^s)\\
F(f) &:= f
\end{array}
\]
 for $f: (A,m_A, u_A, \Delta_A, \ep_A) \to (B,m_B, u_B, \Delta_B, \ep_B) \in \mathcal{F}$. Here, $U^{A,A}, \;\nu_A,\; U_{A,A},$ and $\rho_A^s$ are given in \cref{eq:Uupper,eq:Hs-act,eq:Ulower,eq:Hs-coact}, respectively.
 
 \begin{claim}
 \label{claim:FrobAlg.F}
 The assignment $F$ above yields a functor from $\mc{F}$ to ${\sf FrobAlg}(\MH)$.
 \end{claim}
 
 \begin{proof}
 Define $(A, \bar{m}_A, \bar{u}_A, \bar{\Delta}_A, \bar{\ep}_A):=F(A,m_A,u_A,\Delta_A,\ep_A)$. By \cref{thm:alg,thm:coalg}, in particular by Claims~\ref{claim:Alg.F} and~\ref{claim:Coalg.F},  $(A, \bar{m}_A, \bar{u}_A) \in {\sf Alg}(\MH)$ and $(A,\bar{\Delta}_A, \bar{\ep}_A)\in {\sf Coalg}(\MH)$. It remains to show that the Frobenius algebra condition \eqref{eq:Frobalg} holds; namely, that
 \[(\bar{m}_A \tenbar \id)(\id \tenbar \bar{\Delta}_A) ~=~ \bar{\Delta}_A \bar{m}_A ~=~ (\id \tenbar \bar{m}_A)(\bar{\Delta}_A \tenbar \id).\]
 (Recall from Section~\ref{sec:corep} that the associativity constraint of $\MH$ is the identity.)
 We will show the left equality; the proof of the right equality is similar. Let $a \tenbar b \in A \tenbar A.$ Then
 \begin{align*}
(\bar{m}_A \tenbar \id)(\id \tenbar \bar{\Delta}_A)(a \tenbar b) &=  (\bar{m}_A \tenbar \id)(\id \tenbar U_{A,A}\Delta_A)(a \tenbar b)\\
&\overset{\eqref{eq:Ulower}}{=}  (\bar{m}_A \tenbar \id)(a \tenbar \ep(b_{1,[1]}b_{2,[1]})b_{1,[0]} \tenbar b_{2,[0]})\\
&\overset{\eqref{eq:Uupper}}{=}  (m_A \tenbar \id)(a \otimes \ep(b_{1,[1]}b_{2,[1]})b_{1,[0]} \tenbar b_{2,[0]})\\
&\overset{\eqref{eq:delta-C}}{=} a\ep(b_{[1]})b_{[0],1} \tenbar b_{[0],2}\\
&\overset{}{=} a b_1 \tenbar b_2,
 \end{align*}
where the last equality is justified by the following computation:
 \begin{align*}
 b_1 \otimes b_2 &=\; \Delta_C(b) 
 &\hspace{-.25in}\overset{\eqref{eq:Hcomod-2}}{=} \Delta_C(\ep(b_{[1]})b_{[0]}) 
 &=\; \ep(b_{[1]})\Delta_C(b_{[0]}) 
 &\hspace{-.25in} =\ep(b_{[1]})b_{[0],1}\otimes b_{[0],2}.
 \end{align*}
 On the other hand, we have 
 \begin{align*}
\bar{\Delta}_A \bar{m}_A (a \tenbar b) &=U_{A,A}\Delta_A m_AU^{A,A}(a \tenbar b) &
&\overset{\eqref{eq:Frobalg}}{=}U_{A,A}(m_A \otimes \id)(\id \otimes \Delta_A)U^{A,A}(a \tenbar b) 
\\
&=U_{A,A}(ab_1 \otimes b_2) &
&\overset{\eqref{eq:Uupper}}{=}U_{A,A}U^{A,A}(ab_1 \tenbar b_2) 
\\
&\overset{\eqref{eq:UMNidentity}}{=}ab_1 \tenbar b_2. 
 \end{align*}
 We have shown that $F$ takes objects to objects. By \cref{thm:alg,thm:coalg}, $F$ is functorial and $F$ takes morphisms of $\mc{F}$ to morphisms of ${\sf FrobAlg}(\MH)$.
 \end{proof}
 
 Now take $(A, ~\bar{m}_A, ~\bar{u}_A, ~\bar{\Delta}_A, ~\bar{\ep}_A) \in {\sf FrobAlg}(\MH)$, and consider the assignment
\[
\begin{array}{rl}
\medskip
 G(A, ~\bar{m}_A, ~\bar{u}_A,~\bar{\Delta}_A, ~\bar{\ep}_A) &:=
(A,  ~\bar{m}_A U_{A,A}, ~\bar{u}_A U_0,  ~U^{A,A} \:\bar{\Delta}_A,~ U^0\: \bar{\ep}_A)\\
G(\bar{f}) &:= \bar{f}
\end{array}
\]
 for $\bar{f}: (A, \bar{m}_A, \bar{u}_A, \bar{\Delta}_A, \bar{\ep}_A) \to (B, \bar{m}_B, \bar{u}_B, \bar{\Delta}_B, \bar{\ep}_B) \in {\sf FrobAlg}(\MH)$. Here, $U_{A,A}$, $U^{A,A}$, $U_0$, and $U^0$ are given in \cref{sec:U}.

\begin{claim}
 \label{claim:FrobAlg.G}
 The assignment $G$ given above yields a functor from ${\sf FrobAlg}(\MH)$ to $\mc{F}$.
 \end{claim}
 
 \begin{proof}
Recall from \cref{sec:U} that the forgetful functor $U:\MH \to \Vec$ is Frobenius monoidal. Then by \ref{prp:monfunctor}(iii), \[G(A, \bar{m}_A, \bar{u}_A, \bar{\Delta}_A, \bar{\ep}_A)=
(U(A),  ~U(\bar{m}_A) U_{A,A}, ~U(\bar{u}_A) U_0,  ~U^{A,A} U(\bar{\Delta}_A), ~U^0 U( \bar{\ep}_A))\] is a Frobenius algebra in $\Vec$.
 Moreover, by \cref{thm:alg,thm:coalg}, in particular by Claims~\ref{claim:Alg.G} and~\ref{claim:Coalg.G}, we know that $(A, m_A, u_A, \Delta_A, \ep_A):=G(A, \bar{m}_A, \bar{u}_A, \bar{\Delta}_A, \bar{\ep}_A)$ is both a right $H$-comodule algebra and a right $H$-comodule coalgebra.  Hence $G$ takes objects to objects.
By \cref{thm:alg,thm:coalg}, $G$ is functorial and $G$ takes morphisms of $\mc{F}$ to morphisms of ${\sf FrobAlg}(\MH)$.
 \end{proof}

Now by Claims \ref{claim:Alginverse} and \ref{claim:Coalg.inverse}, the functors $F$ defined in Claim~\ref{claim:FrobAlg.F} and $G$ defined in Claim~\ref{claim:FrobAlg.G} are mutual inverses. Therefore, these functors yield an isomorphism between the categories $\mathcal{F}:=\mathcal{F}^H$ and ${\sf FrobAlg}(\MH)$.
\end{proof} 

Now consider the following examples.

\begin{example} \label{ex:Frobalg-Hs}
The counital subalgebra $H_s$ is  an object of $\mathcal{F}^H \cong {\sf FrobAlg}(\MH)$, as the unit object of a monoidal category is a Frobenius algebra  in the category; see Section~\ref{sec:unitobj-MH}.
\end{example}

\begin{example}\label{ex:Frobalg-kQ} We showed in \cref{ex:alg-kQ,ex:coalg-kQ} that $\kk Q$ is an $\hay$-comodule algebra and an $\hay$-comodule coalgebra. By definition, $\kk Q$ is an $\hay$-comodule Frobenius algebra if $\kk Q$ is first  a Frobenius $\kk$-algebra.
But it is straightforward to show that under the given algebra and coalgebra structure that $\kk Q$ is a Frobenius $\kk$-algebra if and only if $Q=Q_0$ (i.e., $Q$ is arrowless).  
In this case, $\kk Q \cong \hay_s$ in the category $\mathcal{F}$, so this example reduces to \cref{ex:Frobalg-Hs}. Thus $\kk Q = \kk Q_0$ is a Frobenius algebra in $\mathcal{M}^{\hay}$ by \cref{thm:FrobAlg}.
\end{example}

On the other hand, we provide below an example of a Frobenius algebra over the weak bialgebra $\hay$ that is not the counital subalgebra $\hay_s$.

\begin{example}\label{ex:FrobAlgMatrixAlg}
Take $Q$ to be the quiver below: 
\[Q = \begin{tikzcd}
1 \arrow[r, "p", shift left = 1] & 2 \arrow[l, "p^*", shift left = 1]
\end{tikzcd}\]
Consider the $\kk$-algebra $A=\kk Q/(pp^* - e_1,~p^*p-e_2)$. Then $A$ admits the structure of a Frobenius $\kk$-algebra via 
\[
\begin{array}{rlcrl}
\Delta(e_1)&= e_1 \otimes e_1 + p \otimes p^*,& \quad \quad &\Delta(p)&= e_1 \otimes p + p \otimes e_2,\\
\Delta(p^*)&= p^* \otimes e_1 + e_2 \otimes p^*,&\quad \quad &\Delta(e_2) &= p^* \otimes p + e_2 \otimes e_2,
\end{array}
\]
and $\ep(e_1) = \ep(e_2) = 1$, $\ep(p) = \ep(p^*) =0.$ Namely, $A$ is isomorphic to the matrix algebra Mat$_2(\kk)$ via $e_1,p,p^*,e_2 \mapsto E_{11}, E_{12}, E_{21}, E_{22}$, respectively.
Moreover, we have a coaction $\rho: A \to A \otimes \hay$ given by
\[
\begin{array}{rlcrl}
\rho(e_1) &= e_1 \otimes x_{1,1} + e_2 \otimes x_{2,1}, &\quad \quad 
&\rho(e_2) &= e_1 \otimes x_{1,2} + e_2 \otimes x_{2,2}\\
\rho(p) &= p \otimes x_{p,p} + p^* \otimes x_{p^*,p}, &\quad \quad & 
\rho(p^*) &= p \otimes x_{p,p^*} + p^* \otimes x_{p^*,p^*},
\end{array}
\]
that yields that $A$ belongs to the category $\mathcal{F}$. The coassociativity and counitality of $\rho$ follows from a similar proof to \cref{ex:alg-kQ}. We already know that $A$ is a Frobenius algebra, so it only remains to check that $\kk Q$ is an $\hay$-comodule algebra and an $\hay$-comodule coalgebra. The first condition follows from the proof of \cref{ex:alg-kQ} and by checking that $\rho(p)\rho(p^*)=\rho(e_1)$ and $\rho(p^*)\rho(p)=\rho(e_2)$. The second condition can be checked by hand.
By \cref{thm:FrobAlg}, we can give $A$ the structure of a Frobenius algebra  in $\mathcal{M}^{\hay}$.
\end{example}


\section{Weak quantum symmetry via quantum transformation groupoids} \label{sec:QTG}

In the section, we consider a weak Hopf algebra that arises from a double-crossed product construction: the quantum transformation groupoid $H(L,B,\triangleleft)$ [Definition~\ref{def:QTG}]. Here, $L$ is a Hopf algebra and $B$ is a certain right $L$-module algebra via action $\triangleleft$ [Notation~\ref{not:QTG}]. To achieve a supply of weak quantum symmetries, we produce a monoidal functor from a category of $L$-bicomodules to $H$-comodules [Theorem~\ref{thm:QTG}], thus illustrating the main result of Section~\ref{sec:algebras}. We then present explicit examples of these new weak quantum symmetries.

We assume in this section that $\kk$ is an algebraically closed field. We also fix  standing notation and hypotheses below.

\begin{notation}[$L$, $B$, $e^{(i)}$, $\triangleleft$, $\triangleright$, $\omega$] \label{not:QTG} Let $L$ be a Hopf algebra over $\kk$. 
Let $B$ be a strongly separable $\kk$-algebra with symmetric separability idempotent $e = e^{(1)} \otimes e^{(2)} \in B \otimes B^{op}$, where we omit summation notation. If we use multiple copies of the separability idempotent in a calculation, we will denote the additional copies with primes like so: $e'^{(1)}\otimes e'^{(2)}$. We have the following identities:
\begin{align}
 be^{(1)} \otimes e^{(2)} &= e^{(1)} \otimes e^{(2)}b, \quad \quad \forall b \in B, \label{e-iden-1}\\
   e^{(1)} e^{(2)} &= 1_B. \label{e-iden-2}
\end{align}
Assume that $B$ is a right $L$-module algebra with action
$$  \quad \quad \quad \quad  B \otimes L \to B, \quad \quad b \otimes h \mapsto b \triangleleft h, \quad \quad \forall b \in B, ~h \in L.$$
So, $B^{op}$ is a left $L$-module algebra with action via the antipode $S_L$ of $L$,
\begin{equation}  \label{e-QTGaction}\quad \quad \quad \quad  L \otimes B^{op} \to B^{op}, \quad \quad h \otimes a \mapsto h \triangleright a:= a \triangleleft S_L(h), \quad \quad \forall a \in B^{op}, ~h \in L. \end{equation}
Let $\omega: B \to \kk$ be the nondegenerate trace form of $B$, which is characterized by
\begin{align}
\omega(e^{(1)})e^{(2)} ~=~ e^{(1)} \omega(e^{(2)}) ~=~ 1_B. \label{e-iden-3} 
\end{align}
\end{notation}

We assume the following useful identities:
\begin{align}
\omega((h \triangleright a)b) &~=~ \omega(a(b \triangleleft h)), \label{e-iden-4} \\
e^{(1)} \otimes (h \triangleright e^{(2)}) &~=~ (e^{(1)} \triangleleft h) \otimes e^{(2)}, \label{e-iden-5} 
\end{align}
for $a \in B^{op}$, $b \in B$, and $h \in L$.

\smallskip

Now we consider the following weak Hopf algebra introduced in \cite[Section~2.6]{NV}, but modified slightly here; see also \cite[Example 3, Appendix D]{Nill}. We verify that this is indeed a weak Hopf algebra in the first preprint version of this work (available at arXiv:math/1911.12847v1).

\begin{definition}[$H:=H(L,B,\triangleleft)$] \label{def:QTG}
Let $H:=H(L,B,\triangleleft)$ be the weak Hopf algebra, which is $B^{op} \otimes L \otimes B$ as a $\kk$-vector space, with the following structure:

\smallskip

\hspace{-.2in}
{\renewcommand{\arraystretch}{1.2} 
\begin{tabular}{rlll}
     $\bullet$ & multiplication & & $m((a \otimes h \otimes b) \otimes (a' \otimes h' \otimes b')) ~=~ (h_1 \triangleright a')a \otimes h_2 h_1' \otimes (b \triangleleft h_2')b'$;  \\
     $\bullet$ & unit  & & $u(1_\kk) = 1_{B} \otimes 1_L \otimes 1_B$; \\
     $\bullet$ & comultiplication & &  $\Delta(a \otimes h \otimes b) ~=~ (a \otimes h_1 \otimes e^{(1)}) \otimes ((h_2 \triangleright e^{(2)}) \otimes h_3 \otimes b)$; \\
     $\bullet$ & counit & & $\varepsilon(a \otimes h \otimes b) ~=~ \omega(a(b \triangleleft S_L^{-1}(h)))$;\\
     $\bullet$ & antipode & &  $S(a \otimes h \otimes b) ~=~ b \otimes S_L(h) \otimes a$;
\end{tabular}
}

\noindent for all $a,a' \in B^{op}$, $b, b' \in B$, and $h, h' \in L$. We refer to $H$ as a {\it quantum transformation groupoid} (QTG). 
\end{definition}

\begin{lemma} \label{lem:qtg-subalg}
The source and target counital subalgebras of the QTG $H:=H(L,B,\triangleleft)$ are given, respectively, as follows:
$$H_s = 1_{B} \otimes 1_L \otimes B \qquad \text{and} \qquad H_t = B^{op} \otimes 1_L \otimes 1_B.$$
\end{lemma}

\begin{proof}
This follows from the computations below: For $a \in B^{op}, h \in L, b \in B$, we get
\begin{align*}
\varepsilon_s(a \otimes h \otimes b) 
&= (1_B \otimes 1_L \otimes e^{(1)}) \varepsilon((a \otimes h \otimes b)(e^{(2)} \otimes 1_L \otimes 1_B))\\
&=(1_B \otimes 1_L \otimes e^{(1)})\varepsilon((h_1 \triangleright e^{(2)})a \otimes h_2 \otimes b)\\
&=\omega((h_1 \triangleright e^{(2)})a(b \triangleleft S_L^{-1}(h_2)))(1_B \otimes 1_L \otimes e^{(1)})\\
&\overset{\eqref{e-iden-5}}{=} \omega(e^{(2)}a(b \triangleleft S_L^{-1}(h_2)))(1_B \otimes 1_L \otimes (e^{(1)} \triangleleft h_1))\\
&\overset{\eqref{e-iden-1}}{=} \omega(e^{(2)})(1_B \otimes 1_L \otimes ((a(b \triangleleft S_L^{-1}(h_2))e^{(1)}) \triangleleft h_1))\\
&\overset{\eqref{e-iden-3}}{=} 1_B \otimes 1_L \otimes (a(b \triangleleft S_L^{-1}(h_2))) \triangleleft h_1\\
&= 1_B \otimes 1_L \otimes (a \triangleleft h_1)(b \triangleleft \varepsilon_L(h_2))\\
&= 1_B \otimes 1_L \otimes (a \triangleleft h)b, \quad  \text{and } \\\\
\varepsilon_t(a \otimes h \otimes b) 
&= \varepsilon((1_B \otimes 1_L \otimes e^{(1)})(a \otimes h \otimes b))(e^{(2)} \otimes 1_L \otimes 1_B)\\
&= \varepsilon(a \otimes h_1 \otimes (e^{(1)} \triangleleft h_2)b) (e^{(2)} \otimes 1_L \otimes 1_B)\\
&= \omega(a((e^{(1)}\lt h_2)b \lt S_L^{-1}(h_1) )(e^{(2)}\otimes 1_L \otimes 1_B)\\   
&= \omega(a(e^{(1)}\lt h_2 S_L^{-1}(h_{1,2}))(b \lt S_L^{-1}(h_{1,1}) ))(e^{(2)}\otimes 1_L \otimes 1_B)\\ 
&= \omega(ae^{(1)}(b \lt S_L^{-1}(h) ))(e^{(2)}\otimes 1_L \otimes 1_B)\\ 
&\overset{\eqref{e-iden-1}}{=} \omega(e^{(1)}(b \lt S_L^{-1}(h) ))(e^{(2)}a\otimes 1_L \otimes 1_B)\\
&\overset{\eqref{e-iden-4}}{=} \omega((S_L^{-1}(h) \rt e^{(1)})b )(e^{(2)}a\otimes 1_L \otimes 1_B)\\
&\overset{\eqref{e-QTGaction}}{=} \omega((e^{(1)} \lt h)b )(e^{(2)}a\otimes 1_L \otimes 1_B)\\
&\overset{\eqref{e-iden-5}}{=} \omega(e^{(1)} b )((h \rt e^{(2)})a\otimes 1_L \otimes 1_B)\\
&\overset{\omega \textnormal{ is a trace}}{=} \omega(be^{(1)} )((h \rt e^{(2)})a\otimes 1_L \otimes 1_B)\\
&\overset{\eqref{e-iden-1}}{=} \omega(e^{(1)} )((h \rt (e^{(2)}b))a\otimes 1_L \otimes 1_B)\\
&\overset{\eqref{e-iden-3}}{=} (h \rt b)a\otimes 1_L \otimes 1_B.
\end{align*}

\vspace{-.25in} 

\end{proof}

One avenue of obtaining ``weak quantum symmetries'' is via the following construction. First, let us consider the following monoidal category. 

\begin{definition}[$\LBicomod$] \label{def:Lbicomod}
Let $(\LBicomod, \otimes, \kk)$ be the monoidal category consisting of $L$-bicomodules with the monoidal structure below. Note that if $X \in \LBicomod$, then 
\begin{equation} \label{eq:L-bicomod}
    x_{[0],[-1]} \otimes x_{[0],[0]} \otimes x_{[1]} = x_{[-1]} \otimes x_{[0],[0]} \otimes x_{[0],[1]}.
\end{equation}

For $X, X' \in \LBicomod$, the bicomodule structure on $X\otimes X'$ is determined by 
\begin{align} 
 X \otimes X' \to L \otimes X \otimes X', &\quad x \otimes x' \mapsto x_{[-1]}x'_{[-1]} \otimes x_{[0]} \otimes x'_{[0]}, \label{eq:L-bicomod-left}\\
  X \otimes X' \to X \otimes X' \otimes L, &\quad x \otimes x' \mapsto  x_{[0]} \otimes x'_{[0]} \otimes x_{[1]}x'_{[1]}. \label{eq:L-bicomod-right}
\end{align}
This monoidal category has unit object $\kk$ via the counit of $L$, and has the canonical isomorphisms for associativity and unitality. 
\end{definition}

Now, as an abelian category, 
$\LBicomod$ is the same as the category ${}^L \mathcal{M}^L$ 
defined in Example~\ref{ex:monoidcat}. However, 
they differ as monoidal categories: the tensor product in 
 $\LBicomod$ is $\otimes_\kk$, while the tensor product  in ${}^L\mathcal{M}^L$ is  $\otimes^L$. 

\begin{theorem} \label{thm:QTG}
Consider the assignment $\Gamma: \mathcal{M}^L \to \MH$, where for $X \in  \mathcal{M}^L$, take
\[
\begin{array}{c}
\smallskip
\Gamma(X) = B^{op} \otimes X \otimes B, \quad \text{with}\\
\rho_{\Gamma(X)}(a \otimes x \otimes b) = (a \otimes x_{[0]} \otimes e^{(1)}) \otimes ((x_{[1],1} \triangleright e^{(2)}) \otimes x_{[1],2} \otimes b) \in \Gamma(X) \otimes H,
\end{array}
\]
and for a morphism $f \in \mathcal{M}^L$, take $\Gamma(f) = \id_{B^{op}} \otimes f \otimes \id_B$. Then,
\begin{enumerate}[label=(\roman*), font = \upshape]
    \item $\Gamma$ is a functor;
    \item Let $U': \LBicomod \to \mathcal{M}^L$ denote the forgetful functor which forgets the left $L$-comodule structure. Then the precomposition $\widehat{\Gamma} = \Gamma \circ U': \LBicomod \to \MH$ admits a monoidal structure, where
\[
\begin{array}{c}
\smallskip
\hspace{-.45in} \widehat{\Gamma}_{X,X'}: \widehat{\Gamma}(X)~ \tenbar \; \widehat{\Gamma}(X') \longrightarrow \widehat{\Gamma}(X \otimes X')\\ 
\hspace{1in} (a \otimes x \otimes b) \tenbar (a' \otimes x' \otimes b') \mapsto  (x_{[-1]} \triangleright a')a \otimes x_{[0]} \otimes x_{[0]}' \otimes (b \triangleleft x_{[1]}') b'\\\\
\widehat{\Gamma}_{0}: \unit_{\MH} = H_s \longrightarrow \widehat{\Gamma}(\kk) = \widehat{\Gamma}(\unit_{\LBicomod})\\
 1_B \otimes 1_L \otimes b \mapsto 1_{B} \otimes 1_\kk \otimes b.
\end{array}
\]
\end{enumerate}
\end{theorem}

\begin{proof}
(i) Consider the following calculations:
\begin{align*}
    &(\rho_{\Gamma(X)} \otimes \id)\rho_{\Gamma(X)}(a \otimes x \otimes b)\\
    &=(\rho_{\Gamma(X)} \otimes \id)[(a \otimes x_{[0]} \otimes e^{(1)}) \otimes ((x_{[1],1} \triangleright e^{(2)}) \otimes x_{[1],2} \otimes b)]\\
    &=[a \otimes x_{[0],[0]} \otimes e'^{(1)} \otimes (x_{[0],[1],1} \triangleright e'^{(2)}) \otimes x_{[0],[1],2} \otimes e^{(1)}] \otimes ((x_{[1],1} \triangleright e^{(2)}) \otimes x_{[1],2} \otimes b)\\
    &\overset{\eqref{eq:Hcomod-1}}{=}(a \otimes x_{[0]} \otimes e'^{(1)}) \otimes [(x_{[1],1} \triangleright e'^{(2)}) \otimes x_{[1],2,1} \otimes e^{(1)} \otimes ((x_{[1],2,2} \triangleright e^{(2)}) \otimes x_{[1],2,3} \otimes b]\\
    &= (\id \otimes \Delta)[(a \otimes x_{[0]} \otimes e'^{(1)}) \otimes ((x_{[1],1} \triangleright e'^{(2)}) \otimes x_{[1],2} \otimes b)]\\
    &= (\id \otimes \Delta) \rho_{\Gamma(X)}(a \otimes x \otimes b),
\end{align*}
where the third equality holds by the comodule and the coassociativity axioms. Moreover, 
\begin{align*}
    &(\id \otimes \varepsilon)\rho_{\Gamma(X)}(a \otimes x \otimes b)\\
    &= (\id \otimes \varepsilon)[(a \otimes x_{[0]} \otimes e^{(1)}) \otimes ((x_{[1],1} \triangleright e^{(2)}) \otimes x_{[1],2} \otimes b)]\\
    &=(a \otimes x_{[0]} \otimes e^{(1)})\; \omega((x_{[1],1} \triangleright e^{(2)})(b \triangleleft S_L^{-1}(x_{[1],2})))\\
    &\overset{\eqref{e-iden-5}}{=}(a \otimes x_{[0]} \otimes (e^{(1)} \triangleleft x_{[1],1}))\; \omega(e^{(2)}(b \triangleleft S_L^{-1}(x_{[1],2})))\\
    &\overset{\eqref{e-iden-1}}{=}(a \otimes x_{[0]} \otimes (((b \triangleleft S_L^{-1}(x_{[1],2}))e^{(1)}) \triangleleft x_{[1],1}))\; \omega(e^{(2)})\\
    &\overset{\eqref{e-iden-3}}{=}a \otimes x_{[0]} \otimes ((b \triangleleft S_L^{-1}(x_{[1],2})) \triangleleft x_{[1],1})\\
    &= a \otimes \varepsilon_L(x_{[1]})x_{[0]} \otimes b\\
    &= a \otimes x \otimes b.
\end{align*}
So, $\Gamma(X) \in \MH$.
It is clear that $\Gamma(f)$ is a morphism in $\MH$, and with this, $\Gamma$ is a functor.

\smallskip

(ii) First, we have that $\widehat{\Gamma}_{X,X'}$ is a morphism in $\MH$ due to the following calculation:
\begin{align*}
    &(\widehat{\Gamma}_{X,X'} \otimes \id_H)\rho_{\widehat{\Gamma}(X) \tenbar \widehat{\Gamma}(X')}[(a \otimes x \otimes b) \tenbar (a' \otimes x' \otimes b')]\\
    &\overset{\eqref{eq:rhoMN}}{=}  
    (\widehat{\Gamma}_{X,X'} \otimes \id_H)
    \left[(a \otimes x_{[0]} \otimes e^{(1)}) \tenbar (a' \otimes x'_{[0]} \otimes e'^{(1)}) \right.\\
    &  \hspace{1.5in} \otimes \left. \left( (x_{[1],1} \triangleright e^{(2)}) \otimes x_{[1],2} \otimes b \right)
    \left( (x'_{[1],1} \triangleright e'^{(2)}) \otimes x'_{[1],2} \otimes b' \right) \right]\\
    &= \left[(x_{[0],[-1]} \triangleright a')a \otimes x_{[0],[0]} \otimes x'_{[0],[0]} \otimes (e^{(1)} \triangleleft x'_{[0],[1]})e'^{(1)}\right] \\
     &  \hspace{1.5in} \otimes \left( (x_{[1],1} \triangleright e^{(2)}) \otimes x_{[1],2} \otimes b \right)
    \left( (x'_{[1],1} \triangleright e'^{(2)}) \otimes x'_{[1],2} \otimes b' \right)\\
    &\overset{\eqref{eq:L-bicomod}}{=} \left[(x_{[-1]} \triangleright a')a \otimes x_{[0],[0]} \otimes x'_{[0],[0]} \otimes (e^{(1)} \triangleleft x'_{[0],[1]})e'^{(1)}\right] \\
     &  \hspace{1.5in} \otimes \left( (x_{[0],[1],1} \triangleright e^{(2)}) \otimes x_{[0],[1],2} \otimes b \right)
    \left( (x'_{[1],1} \triangleright e'^{(2)}) \otimes x'_{[1],2} \otimes b' \right)\\
    &\overset{\eqref{e-iden-5}}{=} \left[(x_{[-1]} \triangleright a')a \otimes x_{[0],[0]} \otimes x'_{[0],[0]} \otimes e^{(1)}e'^{(1)}\right] \\
     &  \hspace{1.2in} \otimes \left( (x_{[0],[1],1} x'_{[0],[1]} \triangleright e^{(2)}) \otimes x_{[0],[1],2} \otimes b \right)
    \left( (x'_{[1],1} \triangleright e'^{(2)}) \otimes x'_{[1],2} \otimes b' \right)\\
    &\overset{\textnormal{mult in $H$}}{=} \left[(x_{[-1]} \triangleright a')a \otimes x_{[0]} \otimes x'_{[0]} \otimes e^{(1)}e'^{(1)}\right] \\
     &  \hspace{1.2in} \otimes 
     \left( (x_{[2],1} \triangleright (x'_{[2]} \triangleright e'^{(2)})) (x_{[1]} x'_{[1]} \triangleright e^{(2)}) \otimes x_{[2],2} x'_{[3],1} \otimes 
     (b \triangleleft x'_{[3],2})  b' \right)\\
     &= \left[(x_{[-1]} \triangleright a')a \otimes x_{[0]} \otimes x'_{[0]} \otimes e^{(1)}e'^{(1)}\right]  \otimes 
     \left( (x_{[1]} x'_{[1]} \triangleright e^{(2)}e'^{(2)}) \otimes x_{[2]} x'_{[2]} \otimes 
     (b \triangleleft x'_{[3]})  b' \right)\\
     &\overset{\textnormal{$e$ idemp.}}{=} \left[(x_{[-1]} \otimes a')a \otimes x_{[0]} \otimes x'_{[0]} \otimes e^{(1)}\right]  \otimes 
     \left( (x_{[1]} x'_{[1]} \triangleright e^{(2)}) \otimes x_{[2]} x'_{[2]} \otimes 
     (b \triangleleft x'_{[3]})  b' \right)\\
  &= \left[(x_{[-1]} \triangleright a')a \otimes x_{[0],[0]} \otimes x'_{[0],[0]} \otimes e^{(1)}\right]\\
  &  \hspace{1.2in} \otimes 
     \left( (x_{[0],[1],1} x'_{[0],[1],1} \triangleright e^{(2)}) \otimes x_{[0],[1],2} x'_{[0],[1],2} \otimes 
     (b \triangleleft x'_{[1]})  b' \right)\\
  &\overset{\eqref{eq:L-bicomod-left},\eqref{eq:L-bicomod-right}}{=} \left[(x_{[-1]} \triangleright a')a \otimes (x_{[0]} \otimes x'_{[0]})_{[0]} \otimes e^{(1)}\right]\\
  &  \hspace{1.2in} \otimes 
     \left( (x_{[0]} \otimes x'_{[0]})_{[1],1} \triangleright e^{(2)}) \otimes(x_{[0]} \otimes x'_{[0]})_{[1],2} \otimes 
     (b \triangleleft x'_{[1]})  b' \right)\\
  &= \rho_{\widehat{\Gamma}(X \otimes X')}\left[(x_{[-1]} \triangleright a')a \otimes x_{[0]} \otimes x'_{[0]} \otimes (b \triangleleft x'_{[1]})b'\right]\\
  &= \rho_{\widehat{\Gamma}(X \otimes X')} \widehat{\Gamma}_{X,X'}[(a \otimes x \otimes b) \tenbar (a' \otimes x' \otimes b')].
\end{align*}

Next, we have that $\widehat{\Gamma}_0$ is a morphism in $\MH$ due to the calculation below. By Lemma~\ref{lem:qtg-subalg}, an arbitrary element of $H_s$ is of the form $1_B \otimes 1_L \otimes b$. Now,
\begin{align*}
    (\widehat{\Gamma}_{0} \otimes \id_H)\rho_{H_s}[1_B \otimes 1_L \otimes b]
    & \overset{\textnormal{\S\ref{sec:unitobj-MH}}}{=} (\widehat{\Gamma}_{0} \otimes \id_H)\Delta_H[1_B \otimes 1_L \otimes b]\\
    & = (\widehat{\Gamma}_{0} \otimes \id_H)[(1_B \otimes 1_L \otimes e^{(1)}) \otimes ((1_L \triangleright e^{(2)}) \otimes 1_L \otimes b)]\\
    & = (1_B \otimes 1_\kk \otimes e^{(1)}) \otimes ( (1_L \triangleright e^{(2)}) \otimes 1_L \otimes b)\\
    &= \rho_{\widehat{\Gamma}(\kk)} \widehat{\Gamma}_{0}[1_B \otimes 1_L \otimes b].
\end{align*}

\smallskip

Moreover, we have that $\widehat{\Gamma}_{X,X'}$ satisfies the associativity constraint required for $\widehat{\Gamma}$ to be monoidal (see \cref{def:monfunctor}):
{\small 
\begin{align*}
   &\widehat{\Gamma}_{X,X' \otimes X''}(\id \otimes \widehat{\Gamma}_{X',X''})[(a \otimes x \otimes b) \tenbar (a' \otimes x' \otimes b') \tenbar (a'' \otimes x'' \otimes b'')]\\
    &= \widehat{\Gamma}_{X,X' \otimes X''}[(a \otimes x \otimes b) \tenbar ((x'_{[-1]} \triangleright a'')a' \otimes x'_{[0]} \otimes x''_{[0]} \otimes (b' \triangleleft x''_{[1]})b'')]\\
    &\overset{\eqref{eq:L-bicomod-right}}{=}(x_{[-1],2}x'_{[-1]} \triangleright a'')( x_{[-1],1} \triangleright a')a \otimes x_{[0]}  \otimes x'_{[0],[0]} \otimes x''_{[0],[0]} \otimes (b \triangleleft x'_{[0],[1]} x''_{[0],[1]})(b' \triangleleft x''_{[1]})b''\\
    &= (x_{[0],[-1]}x'_{[0],[-1]} \triangleright a'')( x_{[-1]} \triangleright a')a \otimes x_{[0],[0]} \otimes x'_{[0],[0]} \otimes x''_{[0]} \otimes (b \triangleleft x'_{[1]} x''_{[1],1})(b' \triangleleft x''_{[1],2})b''\\
    &\overset{\eqref{eq:L-bicomod-left}}{=}((x_{[0]} \otimes x'_{[0]})_{[-1]} \triangleright a'')( x_{[-1]} \triangleright a')a \otimes (x_{[0]} \otimes x'_{[0]})_{[0]} \otimes x''_{[0]} \otimes (((b \triangleleft x'_{[1]})b') \triangleleft x''_{[1]})b''\\
    &= \widehat{\Gamma}_{X \otimes X', X''}[(x_{[-1]} \triangleright a')a \otimes x_{[0]} \otimes x'_{[0]} \otimes ((b \triangleleft x'_{[1]})b') \tenbar (a'' \otimes x'' \otimes b'')]\\
    &= \widehat{\Gamma}_{X \otimes X', X''}(\widehat{\Gamma}_{X,X'} \tenbar \id)[(a \otimes x \otimes b) \tenbar (a' \otimes x' \otimes b') \tenbar (a'' \otimes x'' \otimes b'')].
\end{align*}
}
\smallskip
Finally,  $\widehat{\Gamma}_{0}$ satisfies the left unit constraint required for $\widehat{\Gamma}$ to be monoidal (see \cref{def:monfunctor}):
\begin{align*}
   & \widehat{\Gamma}(l_X)\;\widehat{\Gamma}_{\kk,X}\;(\widehat{\Gamma}_0 \tenbar \id)[(1_B \otimes 1_L \otimes b) \tenbar (a' \otimes x' \otimes b)]\\  
   &= \widehat{\Gamma}(l_X)\; \widehat{\Gamma}_{\kk,X}[(1_B \otimes 1_\kk \otimes b) \tenbar (a' \otimes x' \otimes b)]\\
&= a'  \otimes x'_{[0]} \otimes (b \triangleleft x'_{[1]})b'\\
&\overset{\eqref{e-iden-3}}{=} \omega(e^{(2)})(a' \otimes x'_{[0]} \otimes (b \triangleleft x'_{[1]})b'e^{(1)})\\
&\overset{\eqref{e-iden-1}}{=} \omega(e^{(2)}(b \triangleleft x'_{[1]})b')(a' \otimes x'_{[0]} \otimes e^{(1)})\\
&= \omega(\varepsilon_L(x'_{[1]})e^{(2)}((b \triangleleft x'_{[2]})b'))(a' \otimes x'_{[0]} \otimes e^{(1)})\\
&= \omega\left((S_L^{-1}(x'_{[2]})x'_{[1]} \triangleright e^{(2)})((b \triangleleft x'_{[3]})b')\right)(a' \otimes x'_{[0]} \otimes e^{(1)})\\
&\overset{\eqref{e-iden-4}}{=} \omega\left((x'_{[1]} \triangleright e^{(2)})\left(((b \triangleleft x'_{[3]})b')\triangleleft S_L^{-1}(x'_{[2]})\right)\right)(a' \otimes x'_{[0]} \otimes e^{(1)})\\
&= \varepsilon\left((x'_{[1]} \triangleright e^{(2)}) \otimes x'_{[2]} \otimes (b \triangleleft x'_{[3]})b'\right)(a' \otimes x'_{[0]} \otimes e^{(1)})\\
&\overset{\text{mult. in $H$}}{=} \varepsilon\left((1_B \otimes 1_L \otimes b)((x'_{[1],1} \triangleright e^{(2)}) \otimes x'_{[1],2} \otimes b')\right)(a' \otimes x'_{[0]} \otimes e^{(1)})\\
&= \varepsilon\left((1_B \otimes 1_L \otimes b)(a' \otimes x' \otimes b')_{[1]}\right)(a' \otimes x' \otimes b')_{[0]}\\
&\overset{\textnormal{\S\ref{sec:unit-MH}}}= l_{\widehat{\Gamma}(X)}[(1_B \otimes 1_L \otimes b) \tenbar (a' \otimes x' \otimes b')],
\end{align*}
and the right unit constraint holds similarly.
\end{proof}

The monoidal structure on $\widehat{\Gamma}$ is prompted by the multiplication of $H:=H(L,B,\triangleleft)$; see Example~\ref{ex:X=L} below. Indeed, $H$ (with coaction defined by $\Delta_H$) is an algebra in the monoidal category $\mathcal{M}^H$. 
However, $H$ is not necessarily a coalgebra in $\mathcal{M}^H$. So, we ask:

\begin{question}
Under what conditions does the functor $\widehat{\Gamma}$ of Theorem~\ref{thm:QTG} admit a comonoidal structure? Further, when does it admit a Frobenius monoidal structure?
\end{question}

\begin{example}[$X=L$] \label{ex:X=L}
Take any Hopf algebra $L$ with strongly separable module algebra $B$ as in Notation~\ref{not:QTG}, and let $H = H(L,B,\triangleleft)$ be the corresponding QTG. 

By Theorem~\ref{thm:QTG}, $\widehat{\Gamma}: \LBicomod \to \mathcal{M}^H$ is monoidal and so by Proposition~\ref{prp:monfunctor}, we obtain an induced functor $\mathsf{Alg}(\LBicomod) \to \mathsf{Alg}(\mathcal{M}^H)$. By Theorem~\ref{thm:alg}, there is an isomorphism of categories $G: \mathsf{Alg}(\mathcal{M}^H) \to \mathcal{A}^H$. Explicitly, we have (using Claim~\ref{claim:Alg.G}) that
\begin{align} \label{eq:compfunc}
\mathsf{Alg}(\LBicomod) \overset{\text{ind. by $\widehat{\Gamma}$}}{\longrightarrow} &\mathsf{Alg}(\mathcal{M}^H) \overset{G}{\longrightarrow} \mathcal{A}^H
\end{align}
{\small$$ \left(L, m_L, u_L\right) \mapsto  \left(\widehat{\Gamma}(L),\; \widehat{\Gamma}(m_L) \widehat{\Gamma}_{L,L}, \;\widehat{\Gamma}(u_L) \widehat{\Gamma}_0\right) \mapsto  \left(\widehat{\Gamma}(L), \;\widehat{\Gamma}(m_L) \widehat{\Gamma}_{L,L} U_{\widehat{\Gamma}(L), \widehat{\Gamma}(L)}, \;\widehat{\Gamma}(u_L) \widehat{\Gamma}_0 U_0\right).$$}

\vspace{-.05in}

\noindent Note that $L \in \mathsf{Alg}(\LBicomod)$ via $\Delta_L$. Then $\widehat{\Gamma}(L) = B^{op} \otimes L \otimes B = H$ as a $\kk$-vector space. By Theorem~\ref{thm:QTG}, the right $H$-coaction on $\widehat{\Gamma}(L)$ coincides with $\Delta_H$.

We claim that 
\[\left(\widehat{\Gamma}(L), \; \widehat{\Gamma}(m_L) \widehat{\Gamma}_{L,L} U_{\widehat{\Gamma}(L), \widehat{\Gamma}(L)}, \; \widehat{\Gamma}(u_L) \widehat{\Gamma}_0 U_0\right) =(H, m_H, u_H),
\]
that is, under the functor \eqref{eq:compfunc}, $L$ maps to $H$. We verify this as follows.
To verify that $\widehat{\Gamma}(u_L)\widehat{\Gamma}_0U_0=u_H,$ notice that 
\begin{equation*}
\widehat{\Gamma}(u_L)\widehat{\Gamma}_0U_0(1_{\kk}) \overset{\eqref{eq:Ulower}}{=}  \widehat{\Gamma}(u_L)\widehat{\Gamma}_0(1_H) 
= \widehat{\Gamma}(u_L)(1_B \otimes 1_{\kk} \otimes 1_B) 
= 1_B \otimes 1_L \otimes 1_B 
= 1_H.
\end{equation*}
To show that $\widehat{\Gamma}(m_L) \widehat{\Gamma}_{L,L} U_{\widehat{\Gamma}(L),\widehat{\Gamma}(L)} = m_H$, we first note that for any $X \in \LBicomod$, a straightforward computation shows,
\[U_{\widehat{\Gamma}(X),\widehat{\Gamma}(X)}((a \otimes x \otimes b)\otimes(a'\otimes x' \otimes b'))=(a \otimes x \otimes be^{(1)})\tenbar (a' \otimes x_{[0]}' \otimes (e^{(2)} \triangleleft x_{[1]}')b').\]
Therefore,
\begin{align*}
&\widehat{\Gamma}(m_L) \widehat{\Gamma}_{L,L} U_{\widehat{\Gamma}(L),\widehat{\Gamma}(L)}((a \otimes x \otimes b)\otimes(a'\otimes x' \otimes b'))\\
& =\widehat{\Gamma}(m_L)((x_{[-1]} \rt a')a \otimes x_{[0]} \otimes x'_{[0],[0]} \otimes ((be^{(1)})\lt x'_{[0],[1]})(e^{(2)} \lt x'_{[1]})b')\\
& =\widehat{\Gamma}(m_L)\l((x_{[-1]} \rt a')a \otimes x_{[0]} \otimes x'_{[0]} \otimes ((be^{(1)})\lt x'_{[1],1})(e^{(2)} \lt x'_{[1],2})b'\r)\\
& =\widehat{\Gamma}(m_L)\l((x_{[-1]} \rt a')a \otimes x_{[0]} \otimes x'_{[0]} \otimes ((be^{(1)}e^{(2)}) \lt x'_{[1]})b'\r)\\
& \overset{\eqref{e-iden-2}}{=} \widehat{\Gamma}(m_L)\l((x_{[-1]} \rt a')a \otimes x_{[0]} \otimes x'_{[0]} \otimes (b \lt x'_{[1]})b'\r)\\
& =\l(\id \otimes m_L \otimes \id\r)\l((x_{[-1]} \rt a')a \otimes x_{[0]} \otimes x'_{[0]} \otimes (b \lt x'_{[1]})b'\r)\\
& =(x_{[-1]} \rt a')a \otimes x_{[0]} x'_{[0]} \otimes (b \lt x'_{[1]})b'\\
& =(x_{1} \rt a')a \otimes x_{2} x'_{1} \otimes (b \lt x'_{2})b' \quad \quad  \text{ (since the bicoaction is given by }\Delta_L)\\
& = m_H((a \otimes x \otimes b)\otimes(a'\otimes x' \otimes b')).
\end{align*}

\smallskip

\end{example}

\begin{example}[$X=\unit_{\LBicomod} = \kk$] Again, take any $L$ and $B$ as in Notation~\ref{not:QTG}, and let $H(L,B,\triangleleft)$ be the corresponding QTG. Then, $\kk \in {\sf Alg}(\LBicomod)$ via $u_L$. By Theorem~\ref{thm:QTG} we get that $\widehat{\Gamma}(\kk) \in {\sf Alg}(\mathcal{M}^{H(L,B,\triangleleft)}$) with $H(L,B,\triangleleft)$-coaction given by 
$$\widehat{\Gamma}(\kk) \to \widehat{\Gamma}(\kk) \otimes H(L,B,\triangleleft), \quad (a \otimes 1_\kk \otimes b) \mapsto (a \otimes 1_\kk \otimes e^{(1)}) \otimes (e^{(2)} \otimes 1_\kk \otimes b)$$ and
 $\widehat{\Gamma}(\kk) \cong B^{op} \otimes B$ as $\kk$-algebras.
 Indeed, the identification is given by $a \otimes 1_\kk \otimes b$ corresponding to $a \otimes b$. Here,
 \begin{align*}
    &G(m_{\widehat{\Gamma}(\kk)})[(a \otimes 1_\kk \otimes b) \otimes (a' \otimes 1_\kk \otimes b')]\\
    &=m_{\widehat{\Gamma}(\kk)}[(a \otimes 1_\kk \otimes b) \tenbar (a' \otimes 1_\kk \otimes b')]\\
    &\overset{\textnormal{\ref{prp:monfunctor}(i)}}{=}
    \widehat{\Gamma}(m_\kk)\; \widehat{\Gamma}_{\kk,\kk}[(a \otimes 1_\kk \otimes b) \tenbar (a' \otimes 1_\kk \otimes b')]\\
    &=  \widehat{\Gamma}(m_\kk)[a'a \otimes 1_\kk \otimes 1_\kk \otimes b  b']\\
    &=  a'a \otimes 1_\kk \otimes b  b',
\end{align*}

\vspace{-.2in}

\begin{align*}
   G(u_{\widehat{\Gamma}(\kk)})(1_\kk)
    &=u_{\widehat{\Gamma}(\kk)}(1_B \otimes 1_\kk \otimes 1_B)
    \overset{\textnormal{\ref{prp:monfunctor}(i)}}{=}
    \widehat{\Gamma}(u_\kk)\; \widehat{\Gamma}_{0}(1_B \otimes 1_\kk \otimes 1_B) &= 1_{B} \otimes 1_\kk \otimes 1_B.
\end{align*}
\end{example}

\begin{remark}
Note from the example above that $\widehat{\Gamma}(\unit_{\LBicomod}) \cong B^{op} \otimes B$, and by Section~\ref{sec:unitobj-MH} and  Lemma~\ref{lem:qtg-subalg} that $\unit_{\mathcal{M}^{H(L,B, \triangleleft)}} \cong B$. So, if $\widehat{\Gamma}$ is strong monoidal, then $B = \kk$.  
\end{remark}

We end this section with a non-trivial example of a QTG, using $\widehat{\Gamma}$ as in Theorem~\ref{thm:QTG}.

\begin{example} Let $G$ be a finite group, and assume that $\kk$ is a field such that $|G| \in \kk^\times$. Take $L = \kk G$ as a Hopf algebra, and also take $B = \kk G$ as a strongly separable right $L$-module algebra via the adjoint action  $b \triangleleft h = h^{-1}bh$ for $h,b \in G$. Here, a (symmetric) separability idempotent of $B$ is given by $|G|^{-1} \sum_{g \in G} g \otimes g^{-1}$, and $\omega = \varepsilon_{\kk G}$;  in particular, the identities~\eqref{e-iden-4} and~\eqref{e-iden-5} hold.
In this case, the QTG $$H:=H(\kk G, \kk G, \triangleleft_{\text{adj}})$$ is a weak Hopf algebra of dimension $|G|^3$ with structure determined by
\begin{align*}
    m_H((a \otimes h \otimes b) \otimes (a' \otimes h' \otimes b')) &= ha'h^{-1}a \otimes h h' \otimes h'^{-1}bh'b',\\
    u_H(1_\kk) &= 1_{\kk G} \otimes 1_{\kk G} \otimes 1_{\kk G},\\
    \Delta_H(a \otimes h \otimes b) &= |G|^{-1} \textstyle \sum_{g \in G} (a \otimes h \otimes g) \otimes (hg^{-1}h^{-1} \otimes h \otimes b),\\
    \varepsilon(a \otimes h \otimes b) &= 1,\\
    S(a \otimes h \otimes b) &= b \otimes h^{-1} \otimes a,
\end{align*}
for $a,a',h,h',b,b' \in G$. Take $1_H:= u_H(1_\kk)$. Note that when $G$ is non-trivial, we get that $H$ is not a Hopf algebra as $\Delta_H(1_H) \not= 1_H \otimes 1_H$ in this case.  

Now given a right $G$-comodule $V$, we may extend linearly so that $V$ is a right $\kk G$-comodule. Using the antipode, we obtain a left $\kk G$-comodule structure on $V$, and since $\kk G$ is cocommutative, this makes $V$ a $\kk G$-bicomodule.
Further, we can extend the $\kk G$-coactions to $V^{\otimes k}$ for all $k \geq 1$ to obtain a $\kk G$-bicomodule structure on the tensor algebra $T(V)$ so that $T(V) \in {\sf Alg}(\kk G\mhyphen \mathsf{Bicomod})$. 
By  Theorem~\ref{thm:QTG}, $(\kk G)^{op} \otimes T(V) \otimes \kk G \in {\sf Alg}(\mathcal{M}^H)$, and so by Theorem~\ref{thm:alg}, it is also a right $H$-comodule algebra (i.e., in the category $\mathcal{A}^H$). Similarly, we can take any ideal $I$ of $T(V)$ that is stable under the $\kk G$-coactions on $T(V)$ to get that $A:=T(V)/I$ belongs to ${\sf Alg}(\kk G\mhyphen\mathsf{Bicomod})$ and $(\kk G)^{op} \otimes A \otimes \kk G \in {\sf Alg}(\mathcal{M}^H) \;(\cong \mathcal{A}^H)$.
\end{example}

\bibliographystyle{alpha}
\bibliography{biblio}{}

\end{document}